\documentclass[final,5p,times]{elsarticle}
\usepackage{amsmath,amssymb,amsfonts}
\usepackage{algorithmic}
\usepackage{graphicx}
\usepackage{microtype}
\usepackage{mathrsfs}  
\usepackage{blindtext}
\usepackage{textcomp}
\usepackage{xcolor}
\usepackage{subcaption}
\usepackage{balance}

\DeclareMathAlphabet{\mathcal}{OMS}{cmsy}{m}{n}

\def\BibTeX{{\rm B\kern-.05em{\sc i\kern-.025em b}\kern-.08em
    T\kern-.1667em\lower.7ex\hbox{E}\kern-.125emX}}

\definecolor{my_orange}{RGB}{255, 153, 85}
\definecolor{my_green}{RGB}{0,153,136}
\newtheorem{prop}{Proposition}

\newcommand{\f}{\mathrm{f}}
\newcommand{\s}{\mathrm{s}}

\newcommand{\T}{\top}
\let\d\undefined\newcommand{\d}{\mathrm{d}}
\newcommand{\model}{\mathcal{M}}
\newcommand{\N}{\mathcal{N}}
\newcommand{\U}{\mathcal{U}}
\newcommand{\R}{\mathbb{R}}
\newcommand{\V}{\mathbb{V}}

\newcommand{\E}{\mathbb{E}}
\newcommand{\ny}{m}
\newcommand{\nx}{n}
\newcommand{\np}{q}
\newcommand{\npce}{d}

\newcommand{\nin}{l}
\newcommand{\nquadrature}{N}
\let\vec\undefined\newcommand{\vec}{\mathrm{vec}}
\let\mat\undefined\newcommand{\mat}{\mathrm{mat}}
\newcommand{\var}{\mathrm{var}}
\newcommand{\diag}{\mathrm{diag}}
\newcommand{\cov}{\mathrm{cov}}

\newcommand{\dof}{p}

\newcommand{\param}{\theta}
\newcommand{\paramSupp}{\Theta}
\newcommand{\paramDistr}{\mathbb{P}_{\param}}
\usepackage{mathtools}

\newcommand{\dist}[1]{\mathbb{P}_{#1}}

\newcommand{\outputSupp}{\mathcal{Y}}

\newtheorem{definition}{Definition}
\newtheorem{remark}{Remark}

\graphicspath{{./}{figures/}}

\usepackage{tikz}
	\usetikzlibrary{external}
\usepackage{pgfplots} 
	\pgfplotsset{compat=newest}
	\pgfplotsset{plot coordinates/math parser=false}
	\pgfplotsset{every axis/.append style={font=\large}}
	\pgfplotsset{every axis label/.append style={scale=1.6}}
\usetikzlibrary{pgfplots.statistics}
\newcommand{\mycaption}[2]{\caption[#1]{ #1. \newline \footnotesize{ #2}}}
\newcommand{\revv}[1]{\textcolor{black}{#1}}
\newcommand{\rev}[1]{\textcolor{black}{#1}}

\begin{document}
\journal{Control Engineering Practice}
\begin{frontmatter}

\title{Uncertainty Quantification Methods for Optimal Excitation Design\\ in Parameter Identification}

\address[CR]{Corporate Sector Research and Advance Engineering, Robert Bosch GmbH, Renningen, Germany}

\author[CR]{Kevin Schmidt}
\author[CR]{Nicola Henkelmann}
\author[CR]{Christoph Mark}
\author[CR]{and Johannes von Keler}

\begin{abstract}
Parameter identification is crucial in virtual engineering processes, yet determining appropriate system excitations for identifying specific parameters remains challenging. In practice, extensive experimental programs often fail to generate data with sufficient information content for reliable parameter estimation.
This work presents a systematic approach for deriving optimal excitations by maximizing the global sensitivity of target parameters across the space of possible excitation functions. To address the computational challenge of sensitivity evaluation during optimization, this work develops two complementary approaches based on uncertainty quantification~(UQ) methods. For systems with known mathematical structure, an intrusive polynomial chaos expansion~(PCE) method constructs deterministic surrogate models, enabling rapid sensitivity computation. For black-box models where intrusive approaches are not feasible, a novel non-intrusive method based on optimal transport theory employs Wasserstein distances to quantify sensitivity measures without requiring knowledge of internal system dynamics.

Both methods significantly reduce computational costs, making optimal excitation design practical for complex engineering systems. Both approaches demonstrate their effectiveness on vehicle dynamics models, showing substantial improvements in parameter identification capability. \rev{The benefit for parameter identification is further validated experimentally on a test vehicle and \revv{compared against representative standard industrial maneuvers}.}
\end{abstract}

\begin{keyword}
Global sensitivity analysis, Polynomial chaos expansion, Optimal transport, Stochastic optimal control
\end{keyword}
\end{frontmatter}

\section{Introduction}
\label{sec:Introduction}
\noindent
\revv{System identification is a fundamental challenge in virtual engineering processes, i.e., model-based development workflows in which simulation environments are used for design, calibration, and validation before extensive physical testing, where accurate model parametrization directly impacts the reliability of simulations and control system design~\cite{ljung1999system}.} The choice of system excitation signals critically determines the quality of parameter identification results. However, determining appropriate excitation signals for identifying specific parameters remains a significant challenge~\cite{mehra2003_survey}. In industrial practice, this often relies on extensive experimental programs and expert knowledge accumulated over years of application experience. The consequence is often a conservative approach involving comprehensive measurement catalogues that require substantial time and financial investment during application and calibration phases. Moreover, the information content of individual maneuvers for system identification purposes often remains unclear, leading to potentially redundant or insufficient experimental procedures.

The fundamental issue underlying these challenges is the lack of systematic methods to quantify the sensitivity of system outputs with respect to uncertain parameters under different excitation conditions. \revv{While classical sensitivity analysis methods exist, they are often computationally demanding for complex engineering systems or are not designed to account for the stochastic nature of prior parameter uncertainty and measurement noise in a global sense. This creates a gap between the theoretical understanding of parameter identifiability and practical excitation design for complex engineering systems with non-Gaussian or broad parameter distributions.}
\hspace{0em}\\[1ex]
\textit{State of the Art.}
The design of optimal excitations for system identification has been widely studied in literature~\cite{mehra2003_survey,ljung1999system}. 
\revv{Standard field test maneuvers include pseudo-random binary sequences (PRBS)~\cite{Schoukens2017}, chirp signals~\cite{Elliott2001}, and multi-sine signals~\cite{Pintelon2012}, and are able to cover a broad frequency spectrum~\cite{ljung1999system}.} These signals serve as a basis for many maneuver designs in practice, but do not exploit particular system properties. Thus, they often lack optimality for the specific system at hand.

On the other hand, the Fisher information matrix (FIM) is the theoretical gold standard for quantification of the achievable accuracy in parameter estimates~\cite{mehra2003_survey}. Related methods typically involve formulating the excitation design as an optimal control problem, where the objective is to maximize a scalar function of the FIM, such as its determinant or trace~\cite{Gevers1986}. \revv{While highly effective for linear systems with Gaussian noise, these approaches can be computationally intensive when extended to nonlinear systems or non-Gaussian uncertainties. In addition, FIM-based optimal experiment design methods typically rely on local sensitivity measures evaluated at a nominal parameter point, and thus do not exploit the full prior distribution of parameter uncertainties that is available in many applications~\cite{Fedorov1972,Atkinson2007}. The present work therefore pursues a complementary direction by replacing local, gradient-based sensitivity with global, distribution-aware measures.} 

In the domain of uncertainty quantification~(UQ), global sensitivity analysis has been widely studied~\cite{saltelli2008, lemaitre2010, sudret2008global}. However, to the best of the authors' knowledge, global sensitivity analysis has not yet been applied to optimal excitation design by combining it with system theoretical aspects. In~\cite{streif2014optimal}, an intrusive PCE-based surrogate for optimal design of experiments is presented, focusing on polynomial systems and without taking robustness and disturbances into account. For the latter aspect, a non-intrusive approach can be found in~\cite{paulson2019efficient} and applications to model predictive control~\cite{paulson2015stability,fisher2008}.

More recent advancements have explored the direction of more robust and application-specific inputs, such as Bayesian experimental design  or Reinforcement Learning techniques. This allows considering prior knowledge and process noise, but often leads to high computational costs for optimization in practice~\cite{Ryan2003, Chaloner1995, Conti2014}.
\hspace{0em}\\[1ex]
\noindent \textit{Contributions. }%
The main contribution of this article is the development of a systematic methodology for optimal excitation design in parameter identification. This methodology formulates a stochastic optimal control problem that maximizes global sensitivity measures while accounting for model uncertainties and measurement noise. The combination of UQ methods and system theoretical aspects enables the derivation of excitations that are robust and specific to the identification task. Key features are the incorporation of prior knowledge about the parameter distributions and the possibility to explicitly consider constraints of the input and output signals.

To solve this computationally challenging problem efficiently, this article develops two complementary uncertainty quantification approaches for sensitivity computation. For systems with known mathematical structure, an intrusive polynomial chaos expansion method derives deterministic surrogate models. For black-box models, a novel non-intrusive approach based on optimal transport theory requires no knowledge of the system's internal dynamics. Both methods significantly reduce the computational burden of global sensitivity computation compared to traditional Monte Carlo approaches. This makes optimal excitation design feasible for complex engineering systems such as vehicle dynamics. Beyond simulation studies, an experimental validation on a test vehicle demonstrates that the optimized excitations substantially improve the identifiability of the target parameter under real-world driving conditions.
\hspace{0em}\\[1ex]
\noindent \textit{Structure of the Article. }%
Section~\ref{sec:Problem} begins by introducing the problem formulation and objectives, presenting the stochastic system modeling approach and formulating the optimal excitation problem as a stochastic optimal control problem~(SCOP), which is computationally intractable.

The subsequent sections introduce two UQ schemes to efficiently compute global sensitivity measures. Section~\ref{sec:intrusive_sensi} develops an intrusive method for global sensitivity analysis based on polynomial chaos expansions (PCE). This approach is particularly applicable to systems with known mathematical structure, where it derives surrogate models for linear parameter-varying (LPV) systems and presents an efficient computation of variance-based sensitivity measures.
For black-box models where intrusive methods are not feasible, Section~\ref{sec:nonintrusive_sensi} presents a non-intrusive approach for sensitivity analysis using optimal transport theory.

The chosen UQ method enables an efficient evaluation of the cost function. Hence, the implementation of the SCOP is addressed in Section~\ref{sec:optimization_method}, which covers the parameterization of input signals and presents optimization algorithms for solving the finite-dimensional optimal excitation problem. A recurring illustrative spring-damper example demonstrates the methodology throughout Sections~\ref{sec:intrusive_sensi}--\ref{sec:optimization_method}.
An industry application is provided in Section~\ref{sec:IndustryExample}, focusing on optimal input design for vehicle systems and comparing both intrusive and non-intrusive approaches on a single-track vehicle dynamics model. Finally, the paper concludes with a discussion of the results and computational benefits of the proposed methods.
\hspace{0em}\\[1ex]
\noindent \textit{Notation. }%
\revv{Throughout the paper, the parameter and output spaces are denoted by $\paramSupp$ and $\outputSupp$, and are assumed to be Polish, i.e., separable and completely metrizable topological spaces such as $\R^n$ or $L^p$ for $1 \leq p < \infty$.}
\revv{The spaces of all Borel probability distributions on $\paramSupp$ and $\outputSupp$ are denoted by $\mathcal{P}(\paramSupp)$ and $\mathcal{P}(\outputSupp)$, respectively.}
Let $z = (\param, y) \in \paramSupp \times \outputSupp$ be a random vector in the probability space $(\Omega, \mathcal{B}, \mathbb{P})$, where $\Omega$ is the sample space, $\mathcal{B}$ the Borel $\sigma$-algebra and $\mathbb{P}$ the probability measure.
The marginal probability measures (distributions) of $z$ are the parameter distribution $\dist{\param} \in \mathcal{P}(\paramSupp)$ and the output distribution $\dist{y} \in \mathcal{P}(\outputSupp)$. 
The conditional distribution of $y$ given $\param$ is denoted as $\dist{y|\param}$.
The (conditional) expectation and (conditional) variance of a random variable $y$ w.r.t. a distribution $\nu \in \mathcal{P}(\outputSupp)$ is denoted as~$\E_{\nu}(y)$ and~$\V_\nu(y)$, respectively.

The variance~$\V_\nu[y]\in\R^m$ of a vector of random variables $y\in\R^m$~\revv{denotes the vector of diagonal entries of the covariance matrix $\cov[y]\in\R^{m\times m}$.}
For vectors~$v\in\R^m$ the exponent $v^n$ denotes the element-wise power with~$(v^n)_i=v_i^n$ for all~$i\in[1,m]$. 
The linear operation $\mu=\vec(M)\in\R^{mn\times 1}$ is the column-wise vectorization of a matrix $M\in\R^{m\times n}$ with $\mu^\T=[M_{1i},M_{2i},\hdots,M_{ni}]$ for all~$i\in[1,m]$. The inverse mapping is given by~$\mat(\mu)=M$ for suitable $m,n$.
\section{Problem description and objectives}
\label{sec:Problem}
\revv{The following paragraphs introduce the stochastic modeling approach for the optimal excitation problem (see Section~\ref{ssec:StochModeling}).} \revv{Next, Section~\ref{ssec:OptimalExcitationProblem} states the general solution of this problem in the form of a stochastic optimal control problem.} \revv{Finally, Section~\ref{ssec:Objectives} summarizes the objectives on the resulting formulation from a technical perspective.}
\subsection{Stochastic system modeling}
\label{ssec:StochModeling}
The basis for the proposed method is the description of the system behavior through a model~$\model$, which relates an input signal~$u(t)\in\R^\nin$ and parameters~$\param\in\R^\np$ with output variables $y(t)\in\R^\ny$, by the mapping
\begin{align}
\label{eq:model_general}
y(t) = \model(u(t),\param).
\end{align}
The parameters~$\param$ which shall be identified are interpreted as uncertain quantities and described by a probability measure $\param\sim\paramDistr$ to include prior knowledge about the parameters or physical constraints.
\revv{If further assumptions about the internal structure of the model~\eqref{eq:model_general} apply, efficient UQ methods for solving the optimal excitation problem are applicable.}

Given a certain input trajectory $u(t)$ for $t\in[0,t_\f]$ resulting in model output~$y(t)$ from~\eqref{eq:model_general}, a sensitivity measure $S_{ij}(t)$ quantifies the amount of information in the output signal.
\revv{In addition, the sensitivity measure obeys the property $[S_{ij}]=[y_i]$, i.e. the unit of the sensitivity measure corresponds to the unit of the output $y_i$ under consideration.}
A natural choice is the Sobol-like global sensitivity measure	
\begin{align}
	\label{eq:example_sensitivity}
S_{ij}^2(t)= \E_{\param_j}\left[ \big( \E_{\sim\param_j}[y_i(t)|\param_j]- \E[y_i(t)]\big)^2 \right]
\end{align}
with the dimension~$S(t)\in\R^{\ny\times\np}$ and the property $[S_{ij}]=[y_i]$, i.e. the unit of the sensitivity measure corresponds to the unit of the output under consideration.
More specifically, the quantity~$S_{ij}^2(t)$ is the amount of variance in the output~$y_i$ which can be explained by the uncertain parameter~$\param_j$.
Note that Sobol' indices are typically normalized in literature with respect to the variance of the output signal~\cite{sobol1993,sudret2008global}. 
\revv{The normalization is intentionally omitted due to the unit-property stated above.}
\revv{Further sensitivity measures and their computation are discussed later in Section~\ref{ssec:SensitivityMeasures_Intrusive}.}
\subsection{Optimal excitation problem}
\label{ssec:OptimalExcitationProblem}
In view of the stochastic sensitivity measure $S(t)$, the optimal excitation problem is stated as a stochastic optimal control problem. 
To this end, an objective function and constraints need to be defined.

Instead of formulating the cost functional directly with respect to a desired sensitivity measure such as~\eqref{eq:example_sensitivity}, the cost functional relies on a surrogate sensitivity measure defined as
\begin{align}
	\label{eq:sensitivity_effective}
	\Delta S(t) \coloneqq \max\left\{0,S(t)-S_{\min} (t)\right\},
\end{align}
where \revv{$S_{\min} (t)$ denotes a minimum sensitivity threshold that needs to be exceeded to ensure sufficient information content in the output signal~$y(t)$.}

\begin{remark}
The minimum sensitivity threshold $S_{\min} (t)$ can be interpreted as the error induced by a model-plant mismatch or sensor noise, both of which cannot be captured by the model parameters~$\param$. In practice, $S_{\min} (t)$ can be derived from prior knowledge about the model-plant mismatch or sensor noise, e.g., similar to the signal-to-noise ratio, or estimated by propagating known disturbances to the output.
Thus, the surrogate sensitivity measure \eqref{eq:sensitivity_effective} factors out the impact of the unmodeled effects.
\end{remark}
Note that, if no information about the minimal sensitivity level is available, one can set $S_{\min} (t)=0$ which results in~$\Delta S(t)=S(t)$ since the targeted global sensitivity measures have the property~$S(t)\geq0$ for $t\in[0,t_\f]$.

With the effective sensitivity measure~\eqref{eq:sensitivity_effective} at hand, a cost functional~$J$ is defined as
\begin{align}
	\label{eq:cost_functional}
J = \int_0^{t_\f} g\big(\Delta S(t),u(t)\big) \d t.
\end{align}
In addition to maximizing the sensitivity measure, the cost functional may also penalize the required input energy to achieve a desired excitation. Thus a possible choice of the stage cost~$g$ is a quadratic form such as 
\begin{align}
 g(\Delta S,u) =  \vec^\T(\Delta S) Q\, \vec(\Delta S) -  \, u^\T R u
\end{align}
with quadratic matrices~$Q\in\R^{\ny\np\times\ny\np}$ and~$R\in\mathbb{R}^{\nin\times\nin}$. 
\revv{Different choices of the sensitivity weight $Q$ allow one to generate excitations that are tailored to the identification of certain parameters or outputs.}

An important feature of the optimal excitation problem is the definition of admissible inputs~$u\in\mathbb{U}$ and constraints on the outputs~$y\in\mathbb{Y}$.
The set of admissible inputs~$\mathbb{U}$ is considered as deterministic and may include bounds on the input amplitude, rate constraints or constraints that characterize desired smoothness. 
In addition, periodic inputs or inputs with a certain energy content may be enforced.

The output constraints~$\mathbb{Y}$ may include physical limitations on the outputs~$y(t)$ such as position or velocity bounds.
\revv{Since the output is stochastic, these constraints need to be formulated in a probabilistic manner in the form of chance or value-at-risk constraints.} A possible formulation for an upper bound~$y_i^{\max}$ on the~$i$-th output is given by
\begin{align}
	\label{eq:chance_constraint}
	\mathrm{Pr}\left(y_i(t) > y_i^{\max}\right) \leq \alpha \quad \forall t \in [0,t_\f],
\end{align}
where $\alpha \in (0,1)$ is a predefined risk level. Using Chebyshev's inequality, the chance constraint~\eqref{eq:chance_constraint} can be approximated by the deterministic constraint using the first two moments of the output distribution~\cite{calafiore2006,marshall1960}.
With the above definitions, the optimal excitation problem is given by the~SOCP
\begin{align}
\label{eq:SOCP_input}
 &\max_{u\in\mathbb{U},t\in[0,t_\f]}&  &J\big(\Delta S(t),u(t))&\qquad&\\
 &~~~~~\mathrm{s.\,t.}&  & y(t) = \model(u(t),\param) \nonumber\\
 &&                 & y \in \mathbb{Y} \nonumber.
\end{align}
The above problem is computationally intractable due to the stochastic nature of the output constraints and the input space of infinite dimension.
Finding solutions requires the repeated evaluation of the sensitivity measures~$S(t)$ for different input trajectories~$u(t)$ during the optimization procedure.
\subsection{Objectives}
\label{ssec:Objectives}
The objective of this article is to find input trajectories~$u(t)$ which maximize the information content in the output signal~$y(t)$ w.r.t. the parameters~$\param\sim\paramDistr$ to be identified in view of the SOCP~\eqref{eq:SOCP_input}.
A key aspect of the stochastic modeling of the SOCP is to incorporate prior knowledge about disturbances and model mismatch through the definition of a minimal sensitivity level~$S_{\min} (t)$ in~\eqref{eq:sensitivity_effective}. 
\revv{The selection of this threshold is crucial to ensure that the derived excitations are robust against disturbances and model mismatch.}

The main contribution of this work is to provide methods to efficiently solve the optimal excitation problem~\eqref{eq:SOCP_input} for different assumptions on the model~\eqref{eq:model_general}. In addition, this work identifies a suitable threshold for the minimal sensitivities $S_{\min} (t)$ ensuring the persistence of excitation.
To achieve this goal, this work proposes different methods from uncertainty quantification~(UQ) to efficiently compute the sensitivity measures~$S(t)$ during the optimization procedure:
\begin{itemize}
	\item \revv{An intrusive polynomial chaos expansion (PCE) method with parametric uncertainties and disturbances in Section~\ref{sec:intrusive_sensi} for models~\eqref{eq:model_general} with a known inner structure.}
	\item \revv{A non-intrusive method in Section~\ref{sec:nonintrusive_sensi} based on optimal transport theory for black-box models~\eqref{eq:model_general}.}
\end{itemize}
The contribution of the above aspects goes beyond the proposed optimal excitation problem and may be used in other UQ and control contexts, such as  model order reduction or stochastic optimal control.

\section{The intrusive method for global sensitivity analysis}
\label{sec:intrusive_sensi}
The core idea of an efficient computation scheme for optimal excitations is the \textit{intrusive polynomial chaos expansion} (PCE) method. 
The concept relies on deriving a surrogate model beforehand, which significantly reduces the computational effort for the evaluation of the desired sensitivity measures $S_{ij}$ during the optimization algorithm.
However, further assumptions for the inner model structure of~\eqref{eq:model_general} are required to apply an intrusive PCE~\cite{lemaitre2010}.
This section first introduces the basic notions of PCEs, then states the model class of interest and derives the surrogate model. 
Finally, it gives insight into how to efficiently compute the required sensitivities.

\subsection{Fundamentals of Polynomial Chaos Expansion}
\label{ssec:PCE_Fundamentals}
Consider a probability space $(\Omega, \mathcal{B}(\Omega), \mu)$, where $\Omega$ is the sample space, $\mathcal{B}(\Omega)$ the Borel $\sigma$-algebra on $\Omega$ and $\mu$ the probability measure on $(\Omega,  \mathcal{B}(\Omega))$. In the following, consider a scalar random variable defined on $\Omega= \mathbb{R}$, i.e., $r: \mathbb{R}_{\geq 0} \times \paramSupp \to \mathbb{R}$, where the first argument takes the time and the second one the parameter vector $\theta \in \paramSupp \subseteq \mathbb{R}^q$. 
An approximation of $r$ is then given by the truncated polynomial chaos expansion
\begin{align}
\label{eq:PCE_scalar}
r(t,\param) \approx \sum_{i=1}^{\ell} \phi_i(\param)R_i(t)  = \phi^\T(\param) R(t),
\end{align}
where $R\in\R^I$ are deterministic, so-called chaos coefficients and $\phi(\param)\in\R^I$ is a vector of $\np$-variate polynomials with the property of two components of $\phi$ being orthogonal, i.e.,
\begin{align}
\label{eq:LOTUS}
&\E_{\param}(\phi_i \phi_j) \triangleq \int_{\paramSupp} \phi_i(\param) \phi_j(\param) \d\mu(\param) = \delta_{ij} \E_{\param}(\phi_i^2),
\end{align}
where $\delta_{ij}$ is the Kronecker-delta and $\d\mu(\param)$ the joint probability measure of $\param$.
Since~\eqref{eq:LOTUS} defines an inner product on the underlying function space, the polynomials~$\phi_i$ are referred to as orthogonal polynomials~\cite{sudret2014polynomial}.
For the sake of simplicity, the parameters are assumed to be uniformly distributed throughout this section, i.e., $\dist{\param} = \U(\param^-,\param^+)$ with the expectation $\bar{\param}=\mathbb{E}(\param)=\tfrac{1}{2}(\param^++\param^-)$, restricting the admissible parameter between $[\param^-, \param^+]$.
As a consequence, the polynomials~$\phi_i(\param)$ are products of univariate Legendre-polynomials~\cite{fisher2008}.
In case of a lexicographical truncation scheme, a desired total polynomial degree of~$d$ yields~$\ell\,\np! d!=(\np+d)!$.
Note that the convergence of the PCE~\eqref{eq:PCE_scalar} is exponential in~$d$~\cite{wiener1938,xiu2002} and further details on the truncation error are discussed in~\cite{faulwasser2017}. 
The definition~\eqref{eq:LOTUS} implies $\phi_1=1$ and the following relations for the first moments of a PCE:
\begin{subequations}
	\label{eq:PCE_moments_scalar}
	\begin{align}
	\label{eq:PCE_mean_scalar}
	\E_{\param}(r(t,\param))&= R_1(t) = [1~0~\hdots~0]R(t) \triangleq \mathrm{m}^\T R(t)\\
	\label{eq:PCE_variance_scalar}
	\V_{\param}(r(t,\param)) &= \left[ 0 ~ \E_\param(\phi_2^2)~\hdots~\E_\param(\phi_I^2) \right] R^2(t)\triangleq \mathrm{v}^\T R^2(t).
	\end{align}
\end{subequations}
With the definition~$\Phi_\ny=I_\ny \otimes \phi$, the PCE of a vector valued quantity~$y(t,\theta)\in\R^\ny$ is given by~$y=\Phi^\T_\ny Y$ with an amount of~$\ny \ell$ coefficients. \revv{In addition, analogous relations to~\eqref{eq:PCE_moments_scalar} hold for the moments with $M_\ny=I_\ny \otimes \mathrm{m}$ and $V_\ny = I_\ny\otimes \mathrm{v}$.}
\subsection{Intrusive surrogates for LPV systems with noise}
\label{ssec:intrusive_LPV}
\revv{This subsection derives the intrusive surrogate for the system class of linear parameter-varying (LPV) systems.}
Depending on the context, this is a substantial simplification in comparison to the general model class~\eqref{eq:model_general}. 
On the other hand, this simplification allows explicitly considering modeling mismatches by introducing process noise~$\omega$ and measurement noise~$\nu$. 
\revv{This is useful for quantifying the minimum sensitivity threshold~$S_{\min} (t)$ in~\eqref{eq:sensitivity_effective}, as defined in Section~\ref{ssec:StochModeling}.}

The starting point is a linear parameter-varying (LPV) representation of the dynamic system under consideration
\begin{subequations}
	\label{eq:LPV_stochastic}
	\begin{align}
	 \dot{x} &= A(\theta)x + B(\theta)u + E(\theta) + \omega && \mathrm{for~}t>0,\\
	 \label{eq:LPV_output}
	 y &= C(\theta)x + D(\theta)u + F(\theta) +  \nu && \mathrm{for~}t\geq 0,\\
	 \label{eq:LPV_IC}
	 x(0,\theta)&=x_0(\theta),
	\end{align}
\end{subequations}
where $\omega\sim\N(\bar \omega(\param),\Sigma_w(\param))$ and $\nu\sim\N(\bar \nu(\param),\Sigma_\nu(\param))$ are possibly parameter-dependent Gaussian noise terms with respective mean and covariance. In addition to the quantities defined in the general model~\eqref{eq:model_general}, the system has $x(t,\param)\in\R^\nx$ internal states evolving from an initial condition~\eqref{eq:LPV_IC}.
 
In the LPV-model~\eqref{eq:LPV_stochastic}, there are two different sources of uncertainty. A first step derives the surrogate model for the parameters~$\param$ to be identified using a PCE for the conditional means w.r.t. the noise terms
\begin{subequations}
\label{eq:LPV_PCE_Ansatz}
\begin{align}
\label{eq:PCE_state}
\E_{\omega,\nu}\left(x(t)\right)&=\Phi^\T_\nx(\param) X(t)\\
\label{eq:PCE_output}
\E_{\omega,\nu}\left(y(t)\right)&=\Phi^\T_\ny(\param) Y(t).
\end{align}
\end{subequations}
As a second step, the impact of the noise terms~$\omega$ and~$\nu$ on the output~$y(t)$ is quantified. The dynamics of the chaos coefficients~$X(t)$ of the state is given by the surrogate model
\begin{subequations}
	\label{eq:LTI_surrogate}
	\begin{align}
	\dot{X} &= A'X + B'u + E'&& \mathrm{for~}t>0,~~X(0)=X_0\\
	Y &= C'X + D'u + F' && \mathrm{for~}t\geq 0.
	\end{align}
\end{subequations}
The system matrices are deterministic and obtained by applying a Galerkin projection to~\eqref{eq:LPV_stochastic}. The following set of projections are used to obtain an optimal approximation of the stochastic LPV-system~\eqref{eq:LPV_stochastic}:
\begin{subequations}
	\label{eq:surrogate_matrices}
\begin{align}
	\label{eq:surrogate_matrices_AE}
	A' &= G^{-1}_\nx \E_\param \left(\Phi_\nx^{} A \Phi_\nx^\T \right)\hspace{-1.6ex}&
	 E' &= G^{-1}_\nx \E_\param\left(\Phi^{~}_\nx [E+\bar{\omega}]\right)\\
	\label{eq:surrogate_matrices_BF}
	B' &= G^{-1}_\nx \E_\param \left(\Phi_\nx^{} B \right)\hspace{-1.6ex}& 
	F' &= G^{-1}_\ny \E_\param\left( \Phi^{~}_\ny [F+\bar{\nu}]\right) \\
	\label{eq:surrogate_matrices_CX0}
	C' &= G^{-1}_\ny \E_\param \left(\Phi_\ny^{} C \Phi_\nx^\T  \right)\hspace{-1.6ex}&
	X_0&= G^{-1}_\nx \E_\param\left(\Phi^{~}_\nx x_0\right)\\	
	\label{eq:surrogate_matrices_DG}
	D' &= G^{-1}_\ny \E_\param \left(\Phi_\nx^{} D \right)\hspace{-1.6ex}& 
	G_{(\cdot)} &= \E_\param \big(\Phi_{(\cdot)}^{}  \Phi_{(\cdot)}^\T \big).
\end{align}
\end{subequations}
To quantify the deviation of the surrogate model~\eqref{eq:LTI_surrogate} from the original model~\eqref{eq:LPV_stochastic}, the so-called defects are introduced
\begin{subequations}
	\label{eq:defects}
\begin{align}
	\label{eq:defect_state}
	 \delta_{\dot x}(\param,t) =&~A(\param)\Phi^\T_\nx(\param)X(t) +  B(\param)u(t) + E(\param)+ \bar{\omega}(\param) \nonumber\\ & - \Phi^\T_\nx(\param)\dot{X}(t),\hspace{6.2ex}t>0,\\
	 \label{eq:defect_output}
	 \delta_y(\param,t) =&~ C(\param)\Phi^\T_\nx(\param)X(t) + D(\param)u(t) + 	 F(\param) \nonumber+ \bar{\nu}(\param)\\ & - \Phi^\T_\ny(\param){Y}(t),\hspace{6.2ex}t\geq 0, \\
	 \label{eq:defect_IC}
	 \delta_{x_0}(\param) =&~  \Phi^\T_\nx(\param)X_0 - x_0(\param),\hspace{2ex}t=0
\end{align}
\end{subequations}
which are obtained by inserting the PCE~\eqref{eq:LPV_PCE_Ansatz} into the original model~\eqref{eq:LPV_stochastic} after taking the conditional expectation w.r.t. the noise terms~$\omega$ and~$\nu$.
As stated in the subsequent proposition, the defects~\eqref{eq:defects} are minimal in a weighted $L_1$-norm using the surrogate~\eqref{eq:LTI_surrogate}. 
For this purpose, the following norm applies to a function $z(\param)\in\R^{k\ell}$ on $\param\in\paramDistr$
\begin{align}
\label{eq:L1-norm}
\|z\|_{\phi} \triangleq  \int_{\paramSupp} \left|w^\T \Phi_k (\param)z(\param)\right| \d \mu(\theta)
\end{align}
for any non-zero choice of weights~$w\neq0$ with $w\in\R^{k\ell}$.
This yields the following result for the approximation of the surrogate model~\eqref{eq:LTI_surrogate}:
\begin{prop}[Optimal Approximation] 
	\revv{Given non-singular matrices $\det(G_\nx)\neq0$ and $\det(G_\ny)\neq0$, then the surrogate model~\eqref{eq:LTI_surrogate} globally minimizes the defects~\eqref{eq:defects} for any $w\neq0$ in the norm~\eqref{eq:L1-norm}} 
	for all $t\geq0$ and all $\param\in\paramDistr$.
\end{prop}
\textit{Proof}. Due to the property of orthogonality of the polynomial basis functions~$\phi$, it directly follows $\det(G_\nx)\neq$ and $\det(G_\ny)\neq0$. 
As a consequence, all matrices in~\eqref{eq:surrogate_matrices} are well-defined. \revv{The proof starts by computing the norm of~\eqref{eq:defect_IC} with $w_1\in\R^{\nx\ell}$.} Using the Cauchy-Schwarz inequality, it follows that 
\begin{align}
\label{eq:proof_defect_x0}
\|\delta_{x_0}\|_{\phi} =&~\E_\param\left( \left| w_1^\T \Phi_{\nx} [ \Phi^\T_\nx(\param)X_0 - x_0(\param)] \right|\right) \leq \|w_1 \| \cdot\\
 &  \left\| \E_\param\left( \Phi_{\nx}(\param) \Phi^\T_\nx(\param)]X_0- x_0(\param)] \right) \right\| \triangleq \|w_1 \| \cdot \|v \| \nonumber.
\end{align}
Making use of the definition of~$X_0$ and $G_\nx$ from~\eqref{eq:surrogate_matrices_CX0} and \eqref{eq:surrogate_matrices_DG}, proves that the auxiliary vector $v$ vanishes, due to
\begin{align}
\label{eq:aux_vector_v}
v =  \E_\param \left( \Phi_{\nx}^{}\Phi_{\nx}^\T \right) G_\nx^{-1} \E_\param\left( \Phi_{\nx} x_0(\param) \right) - \E_\param\left( \Phi_{\nx} x_0(\param) \right).
\end{align}
The fact $v=0$ implies $\|\delta_{x_0}\|_{\phi}=0$ due to~\eqref{eq:proof_defect_x0}.
Since \eqref{eq:L1-norm} is a norm with~$\|\delta_{x_0}\|_{\phi}\geq0$, it is globally minimized for any~$w_1\neq0$.
Next, the norm~\eqref{eq:L1-norm} of~\eqref{eq:defect_state} is bounded by
\begin{align}
\label{eq:proof_bound_defect_x}
\|\delta_{\dot x}\|_{\phi}(t) \leq&~\big\|w_2 \big\| \cdot \big\|
 \big( \E_\param(\Phi_{\nx}^{}A\Phi_{\nx}^\T) - \E_\param(\Phi_{\nx}^{}\Phi_{\nx}^\T) A'\big) X(t)\nonumber\\
 & +\big(\E_\param(\Phi_{\nx}^{}B) - \E_\param(\Phi_{\nx}^{}\Phi_{\nx}^\T) B'\big) u(t) \nonumber \\
 & + \E_\param(\Phi_{\nx}^{}(E +\bar w)) - \E_\param(\Phi_{\nx}^{}\Phi_{\nx}^\T) E'
\big\|
\end{align}
 due to the same arguments than above for any choice $w_2\in\R^{\nx\ell}$ and all $t>0$. Again, with the definitions of~$A'$,~$B'$ and~$E'$ from \eqref{eq:surrogate_matrices_AE} and~\eqref{eq:surrogate_matrices_BF}, the upper bound in~\eqref{eq:proof_bound_defect_x} vanishes, which yields~$\|\delta_{\dot x}\|_{\phi}(t)=0$ for all $t>0$ and all~$\param\in\paramDistr$. 
 Repeating the same procedure for~\eqref{eq:defect_output} yields
 \begin{align*}
 \|\delta_{y}\|_{\phi}(t) &\leq~\big\|w_3 \big\| \cdot \big\|
 \big(\E_\param(\Phi_{\ny}^{}C\Phi_{\nx}^\T) - \E_\param( \Phi_{\ny}^{}\Phi_{\ny}^\T) C'\big)X(t)\\
 &+ \big(\E_\param(\Phi_{\ny}^{}D) - \E_\param(\Phi_{\ny}^{}\Phi_{\ny}^\T) D'\big) u(t) \\
 &+ \E_\param(\Phi_{\ny}^{}(F+\bar \nu)) - \E_\param(\Phi_{\ny}^{}\Phi_{\ny}^\T) F'
\big\|.
 \end{align*}
 The above inequality in conjunction with the definitions of~$C'$,~$D'$ and~$F'$ from~\eqref{eq:surrogate_matrices_BF}--\eqref{eq:surrogate_matrices_DG}, shows that $\|\delta_{y}\|_{\phi}(t)=0$ holds for all $t\geq0$ and all $\param\in\paramDistr$ and any weighting vector~$w_3\in\R^{\ny\ell}$. Consequently,~\eqref{eq:defect_output} and~\eqref{eq:defect_state} are globally minimized by~\eqref{eq:LTI_surrogate}, which completes the proof. \hfill$\square$~\\[1ex]
For the computation of the surrogate matrices~\eqref{eq:surrogate_matrices}, the underlying $\np$-dimensional integrals need to be evaluated numerically.
Depending on the number of parameters~$\param$ this can be done by using dense multivariate Gaussian quadrature rules or sparse Smolyak grids~\cite{lemaitre2010, sudret2014polynomial}. 
\revv{For instance, for \eqref{eq:surrogate_matrices_AE}, the integrand $ a(\param)= \Phi_\nx(\param) A(\param) \Phi_\nx^\T(\param)$ yields the result}
\begin{align}
A' = G^{-1}_\nx\left( \sum_{i=1}^{\nquadrature} a(\param_i) w_i + \Delta \right)  
\end{align}
with the quadrature error $\Delta$, the quadrature nodes $\param_i$ and weights $w_i$ for $i=1,\hdots,\nquadrature$.
If the integrand~$a(\param)$ is polynomial in $\param$, i.e. if $A(\param)$ is polynomial in $\param$, there exists a quadrature order $\nquadrature^*$ with vanishing $\Delta=0$ for $\nquadrature\geq \nquadrature^*$.\\[1ex]
\revv{After deriving the surrogate for the expected mean~\eqref{eq:LTI_surrogate}, the second step analyzes the dynamics of the conditional covariances $\cov_w(x(t))=V(t,\param)$ w.r.t. the process noise~$\omega$.}
Given a set of parameters~$\param$, system~\eqref{eq:LPV_stochastic} is linear in view of the process noise~$\omega$. As a consequence the standard relation
\begin{align}
\label{eq:variance_surrogate}
 \dot{V}(t,\param) &= A(\param) V(t,\param) + V(t,\param) A^\T (\param) + \Sigma_w(\param)~~t>0,\nonumber\\
 ~V(0,\param)&=0
\end{align}
holds. 
By superposition of the variance and the conditional mean~\eqref{eq:PCE_state}, the relation
\begin{align*}
\label{eq:state_decomposition}
 x(t,\param,\epsilon) &= \Phi^\T_{\nx} (\param) X(t) + V(t,\param) \epsilon(t)
\end{align*}
with standard Gaussian~$\epsilon\sim\N(0,I_{\nx})$.
\revv{Sensitivity measures require only the variance, not the full PDF.}
When applying the output mapping~\eqref{eq:LPV_output} to~\eqref{eq:state_decomposition} and using the law of total variance, it follows for the output
\begin{align}
\V(y(t)) &= \V_\theta(\E_{w,\nu}(y(t)) + \E_\theta(\V_w(y(t))) + \V_\nu(\E_p(y(t)) \nonumber \\
 &\approx V_{\ny} Y^{2}(t) +  \underbrace{ C(\bar \theta) \,\diag V(t,\bar \theta) + \diag \Sigma_\nu (\bar \theta) }_{\sigma^2_y(t)}
\end{align}
for all $t\geq0$.
This makes it possible to combine the propagated variance of the process noise~$\omega$ and measurement noise~$\nu$ to the variances $\sigma^2_y$.

Consequently, it is sufficient to focus on the conditional mean~\eqref{eq:PCE_output} of the output for the design of optimal input trajectories. In addition, the impact of model uncertainties and measurement noise is accounted for by the choice
\begin{align}
\label{eq:sensitivity_minimal_surrogate}
S_{\min}(t)= \sigma_y(t) 1_\np^\top \in\R^{\ny\times \np}
\end{align}
for the minimal sensitivities in~\eqref{eq:sensitivity_effective}, where $1_\np$ is a vector of length~$\np$ containing ones.
\subsection{\revv{Global sensitivity measures based on PCE-surrogates}}
\label{ssec:SensitivityMeasures_Intrusive}
\revv{As motivated in Section~\ref{sec:Problem}, variance-based global sensitivity measures shall be used for the design of optimal input signals for system identification.}
Typically, first-order~$\mathrm{SU}^{}_{j}(y_i(t))$ and total-order Sobol' indices~$\mathrm{SU}^\mathrm{T}_{j}(y_i(t))$ are used for this purpose~\cite{sudret2008global}.
The first type captures the fraction of the variance of the output~$y_i$ which can be explained by the uncertainty of the parameter~$\param_j$ \textit{alone}.
\revv{On the other hand, the total-order indices reflect the variance of the output $y_i$ caused by $\param_j$, including \textit{all} interactions with the other parameters $\param_{\sim j}$~\cite{saltelli2008}.}
\revv{This section introduces both types of sensitivity measures and afterwards states how to compute them using PCEs.}\\[1ex]
\noindent\textit{First-order sensitivity measures}.
 For all time~$t\geq0$ with positive overall variance~$\var_\theta(y_i(t))>0$, the first-order Sobol' indices are defined by
\begin{align}
\label{eq:Sobol_first_order}
\mathrm{SU}^{}_{j}(y_i(t)) = \frac{\V_{\param_j}\left(\E_{\sim\param_j}(y_i(t))\right)}{\V_\param (y_i(t))} \triangleq \frac{S^{2}_{ij}(t)}{\V_\param(y_i(t))} . 
\end{align}
The notation~$\E_{\sim\param_j}$ is the conditional mean w.r.t. all parameter besides $\param_j$. 
Since $\mathrm{SU}^{}_{j}$ is dimensionless, this contradicts the initial claim from Section~\ref{sec:Problem}, that the chosen sensitivity measure shall be of the same unit than the output under consideration. 
\revv{However, the denominator of the equation above is a meaningful choice for the desired sensitivity measure $S^{2}_{ij}$.}
Exploiting the properties of PCE, allows to compute the quantities~$\mathrm{SU}_{j}(y_i)$ and~$S_{ij}$ very efficiently by virtue of a single simulation of the surrogate~\eqref{eq:LTI_surrogate} -- without any sampling.
From the definition~\eqref{eq:Sobol_first_order} an index-set~$\mathscr{I}_j$ can be determined from the choice of basis function~$\phi(\param)$ for all parameters $\theta_j$ with $j=1,\hdots,\np$
\begin{align}
\label{eq:IdexSet_first_order}
\mathscr{I}_j = \left\{1\leq i \leq \ell \left| \frac{\partial \phi_i(\param)}{\partial \param_j} \neq 0~\wedge~\frac{\partial \phi_i(\param)}{\partial \param_k} = 0,~\forall j\neq k \right. \right\}.
\end{align}
The criterion above checks for each component~$\phi_i(\param)$ of the $\np$-variate basis function if it depends on the parameter~$\param_j$ alone. Only choosing the respective components of the variance mapping~\eqref{eq:PCE_variance_scalar} to be non-zero by
\begin{align}
\label{eq:selection_first_order}
(\mathrm{v}_j)_i = \begin{cases}
	\E_\param(\phi_i^2) & \mathrm{if~} i\in\mathscr{I}_j \\
	0 & \mathrm{otherwise}
\end{cases}
\end{align}
with~$\mathrm{v}_j\in\R^\ell$ allows defining the first-order sensitivity of all outputs with regard to the parameter~$\param_j$
\begin{align}
\label{eq:sensitivity_measure_first_order}
S_{(\cdot)j}(t)  &= \sqrt{(I_\ny\otimes\mathrm{v}^\top_j) Y^2(t) }.
\end{align}
\noindent\textit{Total-order sensitivity measures}. \revv{Since the computation of separate higher order Sobol' indices requires computing $2^\np-1$ indices, the interference of the different parameters is condensed using total-order Sobol' indices~\cite{saltelli2008}.}
\begin{align}
\label{eq:Sobol_total_order}
\mathrm{SU}^\mathrm{T}_{j}(y_i(t)) = \frac{\E_{\sim\param_j}\left(\V_{\param_j} (y_i(t))\right)}{\V_\param(y_i(t))} \triangleq \frac{({S}^\mathrm{T}_{ij})^2(t)}{\V_\param( y_i(t))}. 
\end{align}
Similar as above, this allows to derive an alternative unit-conserving sensitivity measure~${S}^\mathrm{T}_{ij}(t)$. Evaluating~\eqref{eq:Sobol_total_order} for a PCE of the output~$y_i(t)$ leads to the index set
\begin{align}
\mathscr{I}^\mathrm{T}_j = \left\{1\leq i \leq \ell \left| \frac{\partial \phi_i(\param)}{\partial \param_j} \neq 0 \right. \right\}.
\end{align}
\begin{remark}
By comparing the equation above with~\eqref{eq:IdexSet_first_order}, it follows that~$\mathscr{I}_j \subseteq \mathscr{I}^\mathrm{T}_j$ since also interaction terms are allowed in the total-order framing. Analogously than in~\eqref{eq:selection_first_order}, a selection vector~$v^\mathrm{T}_j\in\R^\ell$ makes it possible to compute a relation similar to~\eqref{eq:sensitivity_measure_first_order} for the sensitivity measures~$S^\mathrm{T}_{ij}(t)$ based on the PCE coefficients~$Y(t)$.
\end{remark}
\begin{figure}[tb]
	\centering
	\footnotesize
	\includegraphics[width=\linewidth]{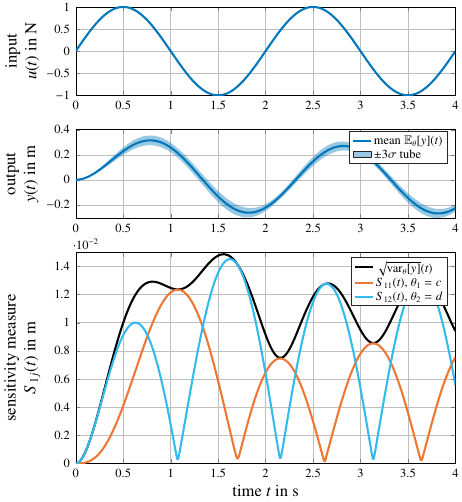}
	\mycaption{Sensitivities of the spring damper system}{The mean output and the overall uncertainty~(middle) for a sinusoidal input~(top) are computed by an intrusive surrogate. The contribution of the parameters to the variance are shown in the lower graph.}
	\label{fig:IntrusiveSensitivitiesSpringDamper}
\end{figure}
\noindent\textit{Summary}. Given a specific trajectory of the input~$u(t)$, the relevant properties of stochastic models~\eqref{eq:LPV_stochastic} can be evaluated by a simulation of the intrusive surrogate model. 
In particular, the simulation of~\eqref{eq:LTI_surrogate} with~$\nx\ell$ states. After that, the computation of the required sensitivities~$S_{ij}(t)$ is straightforward using~\eqref{eq:sensitivity_measure_first_order} for all $j=1,\hdots,\np$.
In addition, a numerical integration of the system~\eqref{eq:variance_surrogate} with $\nx^2$ states evaluated at $\bar\param$ gives the quantity~$S_{\min}(t)$ by virtue of~\eqref{eq:sensitivity_minimal_surrogate}.
Combining both by means of~\eqref{eq:sensitivity_effective} gives the effective sensitivities $\Delta S(t)$ required for the stochastic optimal control problem~\eqref{eq:SOCP_input}.
\subsection{Illustrative example for sensitivities}
\label{sec:intrusive:sens_example}
\noindent As an illustrative example how to compute the sensitivities for a given input signal, consider a spring-damper system governed by the dynamics
\begin{align}
	\label{eq:spring_damper_example}
d \dot{x}(t) &=  -c x(t) + u(t),~~~t>0,~~~x(0) = 0
\end{align}
with the input signal $u(t) = u_0\sin(2\pi f_0 t)$ being a given sinusoidal force with amplitude~$u_0=1.0\,\mathrm{N}$ and frequency~$f_0=0.5\,\mathrm{Hz}$.
In the equation above, $x(t)$ denotes the displacement, $c$ the spring constant and $d$ the damping constant.
Both parameters follow uniform distributions with $\np=2$ and
\begin{align*}
c&\sim \mathcal{U}(1.8\,\mathrm{N/m},\, 2.2\,\mathrm{N/m}), \\
d&\sim \mathcal{U}(0.9\,\mathrm{Ns/m},\, 1.1\,\mathrm{Ns/m}).
\end{align*}
In addition, the displacement~$x(t)$ is considered as the scalar output of the system with~$\ny=1$.
Using a PCE-order of~$\npce=3$, the resulting surrogate model~\eqref{eq:LTI_surrogate} has~$\ell\nx=\ell\ny=10$ states and outputs. 
To illustrate the impact of the disturbances, the example further includes a measurement noise with variance $\Sigma_\nu = 0.007\,\mathrm{m}^2$ and no process noise, i.e., $\Sigma_w = 0$.
\begin{figure}[tb]
	\centering
	\footnotesize
	\vspace{1.5ex}
	\includegraphics[width=\linewidth]{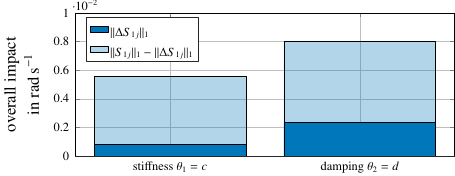}
	\mycaption{Overall sensitivity scores}{The total height indicates the total sensitivity score of the spring damper system, the dark blue part the effective one.}
	\label{fig:IntrusiveSensitivitiesRakingSpringDamper}
\end{figure}

The expected value of the system's displacement is visualized in Fig.~\ref{fig:IntrusiveSensitivitiesSpringDamper}~(middle). The shaded area indicates the overall uncertainty w.r.t. $\param$ given by the standard deviation computed from the conditional variance~\eqref{eq:PCE_variance_scalar}, which is shown as the black line in lower graph in Fig.~\ref{fig:IntrusiveSensitivitiesSpringDamper}. The resulting sensitivities are shown in Fig.~\ref{fig:IntrusiveSensitivitiesSpringDamper}~(bottom) and computed in first-order sense using~\eqref{eq:sensitivity_measure_first_order}. \revv{Peak sensitivities of the spring constant~$c$ occur at the extrema of the sinusoidal input, whereas the damping constant~$d$ has its highest impact during the zero-crossings of the input, where the velocity is maximal.}

To illustrate how to condense the sensitivities into a single score per parameter, the 1-Norm of the time-series condenses the sensitivity scores similar to the cost functional~\eqref{eq:cost_functional}. The results are visualized in Fig.~\ref{fig:IntrusiveSensitivitiesRakingSpringDamper}. The total height of each bar indicates the total sensitivity score~$\|S_{1j}\|_1$, whereas the dark blue part reflects the effective sensitivity~$\|\Delta S_{1j}\|_1$. As can be seen from~\eqref{eq:sensitivity_minimal_surrogate}, the minimal sensitivity~$S_{\min}$ is constant over time here, since the output mapping is deterministic $C(\param)=1$, and purely determined by the measurement noise. The spring constant~$c$ has a higher overall impact on the output uncertainty than the damping constant~$d$.

\section{\revv{The non-intrusive method for global sensitivity analysis}}
\label{sec:nonintrusive_sensi}

This section introduces an approach to non-intrusive global sensitivity analysis using \revv{ optimal transport, initially proposed by \cite{borgonovo2024global} for first-order sensitivity indices.}
This method is particularly useful when dealing with complex models where intrusive methods may not be feasible.
\revv{The section begins with a brief overview of optimal transport theory, followed by the formulation of first-order sensitivity measures as proposed by \cite{borgonovo2024global}. Finally, the optimal transport sensitivity measures are extended to the total-order case.}
\subsection{Introduction to optimal transport}

Given a non-negative lower semi-continuous function (transportation cost) $\kappa : \outputSupp \times \outputSupp \to \R_{\geq 0}$, the optimal transport discrepancy between two probability distributions $\mathbb{P}, \mathbb{Q} \in \mathcal{P}(\outputSupp)$ is defined as
\begin{align}
	D_{\kappa}(\mathbb{P}, \mathbb{Q}) \coloneqq \inf_{\gamma \in \Gamma(\mathbb{P}, \mathbb{Q})} \int_{\outputSupp \times \outputSupp} \kappa(y_1, y_2) d\gamma(y_1, y_2), \label{eq:OT:discrepancy}
\end{align}
where $\Gamma(\mathbb{P}, \mathbb{Q})$ is the set of all joint probability distributions over $\outputSupp \times \outputSupp$ with marginals $\mathbb{P}$ and $\mathbb{Q}$, also known as \textit{couplings} or \textit{transport plans} \cite{villani2008optimal}.

When $\kappa(y_1, y_2) = \lvert y_1 - y_2 \rvert^p$ for some $p \in [1, \infty]$, the optimal transport discrepancy \eqref{eq:OT:discrepancy} collapses to the celebrated type-$p$ Wasserstein distance
\begin{align}
	W_p(\mathbb{P}, \mathbb{Q}) \coloneqq \left(D_{\lvert y_1 - y_2 \rvert^p}(\mathbb{P}, \mathbb{Q}) \right)^{1/p}. \label{eq:OT:Wasserstein}
\end{align}
In general, closed-form expressions of \eqref{eq:OT:Wasserstein} are out of reach. 
However, one prominent exception is the type-2 Wasserstein distance, which according to \cite{gelbrich1990formula} can be expressed as 
\begin{align}
W^2_2(\mathbb{P}, \mathbb{Q}) \coloneqq W^2_\mathrm{B}(\mathbb{P}, \mathbb{Q}) + \Psi(\mathbb{P}, \mathbb{Q}). \label{eq:OT:gelbrich}
\end{align}
The first part is the so-called Bures metric \cite{bures1969extension}
\begin{align}
W^2_\mathrm{B}(\mathbb{P}, \mathbb{Q}) &\coloneqq \mathscr{M}^2(\mu_\mathbb{P}, \mu_\mathbb{Q}) + \mathscr{V}^2(\Sigma_\mathbb{P}, \Sigma_\mathbb{Q}) \label{eq:OT:Bures} 
\end{align}
which computes the deviation between the first two moments
\begin{align*}
\mathscr{M}^2(\mu_\mathbb{P}, \mu_\mathbb{Q}) & \coloneqq \lVert \mu_\mathbb{P} - \mu_\mathbb{Q} \rVert_2^2  \\
\mathscr{V}^2(\Sigma_\mathbb{P}, \Sigma_\mathbb{Q}) &  \coloneqq  \mathrm{tr}\left( \Sigma_\mathbb{P} + \Sigma_\mathbb{Q} - 2\left( \Sigma^{1/2}_\mathbb{Q} \Sigma_\mathbb{P} \Sigma^{1/2}_\mathbb{Q} \right)^{1/2} \right),
\end{align*}
while the second part denotes a non-negative residual term $\Psi(\mathbb{P}, \mathbb{Q}) \geq0$ that is positive for distributions with higher moments. In general it holds that $W^2_2(\mathbb{P}, \mathbb{Q}) \geq W^2_\mathrm{B}(\mathbb{P}, \mathbb{Q})$.
\begin{remark}
	Under the assumption that both, $\mathbb{P}$ and $\mathbb{Q}$ are elliptical distributions with the same density generator, Gelbrich~\cite{gelbrich1990formula} proved that $\Psi(\mathbb{P}, \mathbb{Q}) \equiv 0$.
	Intuitively, this follows from the fact that elliptical distributions are fully specified by their first two moments, i.e., the mean $\mu$ and the covariance $\Sigma$, in which case higher moments do not exist.
\end{remark}
The next step draws the connection between the optimal transport discrepancy and global sensitivities.
\subsection{First-order optimal transport sensitivities}
\label{subsec:first_order_wasser}
\revv{For simplicity, let the parameter vector $\theta$ be scalar, i.e., the parameter distribution $\dist{\theta}$ is univariate}~\footnote{In the multivariate case, it is typically assumed that the parameters $\theta_1, \ldots, \theta_q$ are mutually independent. Thus, the joint parameter distribution $\dist{\param}$ is defined as the product of its marginals $\dist{\param} = \prod_{i=1}^{q} \dist{\theta_i}$. Consequently, the analysis presented in this section can be repeated for each parameter $i = 1, \ldots, q$ independently.}.
The \textit{first-order global optimal transport sensitivity index} of $\param \sim
\dist{\param}$ with respect to $y \sim \dist{y}$ along \eqref{eq:OT:discrepancy} is defined as in \cite{borgonovo2024global}
\begin{align}
	\xi(\param,y) \coloneqq \E_{\param} \left( D_{\kappa}(\dist{y}, \dist{ y|\param }) \right). \label{eq:OT:sensitivity}
\end{align}
Assuming positivity of the transport cost $\kappa$ under two independent random vectors $y_1, y_2 \sim \dist{y}$, i.e, $\E_{y}(\kappa(y_1,y_2)) > 0$, a normalization of \eqref{eq:OT:sensitivity} follows as
\begin{align}
\iota(\param,y) \coloneqq \left(\E_{y}(\kappa(y_1,y_2)) \right)^{-1} \xi(\param,y). \label{eq:OT:normalized_sensitivity}
\end{align}
\begin{remark}
	The normalization $\E_{y}(\kappa(y_1,y_2))$ in \eqref{eq:OT:normalized_sensitivity} can be understood as a generalization of the variance-based normalization from classical sensitivity analysis methods.
\end{remark}
Next, consider the type-2 Wasserstein distance with corresponding global sensitivity index
\begin{align}
\xi_{W^2_2}(\param,y) &\coloneqq \E_{\param} \left( W^2_2(\dist{y}, \dist{ y|\param } ) \right) \nonumber\\
&\overset{\eqref{eq:OT:gelbrich}}{=} \E_{\param} \left( W^2_\mathrm{B}(\dist{y}, \dist{ y|\param }) \right) + \E_{\param} \left(\Psi(\dist{y}, \dist{ y|\param }) \right) \nonumber\\
&\:\geq \E_{\param} \left( W^2_\mathrm{B}(\dist{y}, \dist{ y|\param}) \right) = \xi_{W_\mathrm{B}^2}(\param,y), \label{eq:OT:Wasserstein_sensitivity} 
\end{align}
where the inequality holds as equality if both $\dist{y}$ and $\dist{y|\param}$ are elliptical. 
In view of the normalized index \eqref{eq:OT:normalized_sensitivity}, it amounts to substitute $\kappa(y_1, y_2) = \lvert y_1 - y_2 \rvert^2$ and $\xi_{W_\mathrm{B}^2}(\param,y)$, resulting in
\begin{align}
\iota_\mathrm{B}(\param,y) &\coloneqq (\underbrace{\E_{y}(|y_1-y_2|^2)}_{= 2\mathrm{tr}(\Sigma_{\dist{y}})} )^{-1} \xi_{W_\mathrm{B}^2}(\param,y). \label{eq:OT:normalized_sensitivity_wasserstein}
\end{align}

\begin{figure}[b]
	\centering
	\includegraphics{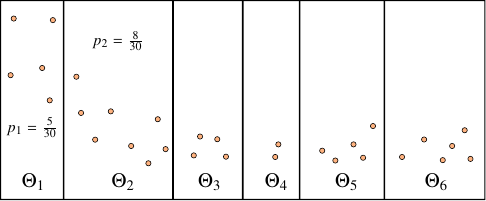}
	\mycaption{Exemplary scatter plot}{Scatter plot with $N_\s = 30$ samples $(\param_i, y_i)$ divided into $M = 6$ bins. $p_1$ and $p_2$ denote the frequencies of $x_i$ being located in bins $\paramSupp_1$ and $\paramSupp_2$, respectively.}
	\label{fig:scatter_plot}
\end{figure}

\begin{remark}
	\label{remark:type2_first_order}
The type-2 Wasserstein sensitivity can be related to variance-based sensitivities \cite{borgonovo2024global}, i.e., first order Sobol indices. By substituting the definition of $\xi_{W_\mathrm{B}^2}(x,y)$ along with \eqref{eq:OT:Bures} into \eqref{eq:OT:normalized_sensitivity_wasserstein}, one obtains the separation 
\begin{align*}
	\iota_\mathrm{B}(\param,y) &= \frac{\mathscr{M}^2( \mu_{\dist{y}}, \mu_{\dist{y|\param}}) + \mathscr{V}^2(\mu_{\dist{y}}, \mu_{\dist{y|\param}})}{2 \mathrm{tr}(\Sigma_{\dist{y}})},
\end{align*}
where the mean deviation (advective part) $\mathscr{M}^2( \mu_{\dist{y}}, \mu_{\dist{y|\param}})$ is connected to the first order Sobol' index as
\begin{align*}
\iota_S(\param,y) &= \frac{\mathscr{M}^2( \mu_{\dist{y}}, \mu_{\dist{y|\param}})}{\mathrm{tr}(\Sigma_{\dist{y}})}.
\end{align*}
\end{remark}
Note that each quantity \eqref{eq:OT:sensitivity} - \eqref{eq:OT:normalized_sensitivity_wasserstein} requires the true (but unknown) parameter and output distribution to be evaluated and thus, an approximation scheme is required.
\subsection{Bin-wise sample-average approximation}
\revv{From now on, the analysis assumes access to $N_\s \in \mathbb{N}$ parameter/output pairs $\{(\param_i, y_i)\}_{i=1}^{N_\s}$, e.g., obtained from simulation or real-world measurements.}
Thus, the empirical counterparts $\dist{\param}^{N_\s} = N^{-1}_\s \sum_{i=1}^{N_\s} \delta_{\param_i}$ and $\dist{y}^{N_\s} = N^{-1}_\s \sum_{i=1}^{N_\s} \delta_{y_i}$ replace the parameter and output distributions $\dist{\param}$ and $\dist{y}$. 
\begin{definition}
	Given an empirical distribution $\dist{y}^{N_\s}$, the sample mean and variance are defined as
	\begin{align}
	\label{eq:OT:sample_mean_variance}
	\hat{\mu}_{y} \coloneqq N^{-1}_\s \sum_{i=1}^{N_\s} y_i, \quad \hat{\Sigma}_{y} \coloneqq N_\s^{-1} \sum_{i=1}^{N_\s} (y_i - \hat{\mu}_{y})^2.
	\end{align}
	Consider a non-empty subset of the parameter space $\paramSupp_j \subset \paramSupp$ with cardinality $|\paramSupp_j| = N_j \geq 1$. The conditional sample mean and variance of $y_i$ given $\param_i \in \paramSupp_j$ are defined as
	\begin{align*}
	\hat{\mu}_{y|\paramSupp_j} \coloneqq \sum_{\param_i \in \paramSupp_j} \frac{y_i}{N_j}, \quad \hat{\Sigma}_{y|\paramSupp_j} \coloneqq \sum_{\param_i \in \paramSupp_j} \frac{(y_i - \hat{\mu}_{y|\paramSupp_j})^2}{N_j}.
	\end{align*}
\end{definition}
In order to estimate \eqref{eq:OT:Wasserstein_sensitivity} from data, a first step partitions the parameter space $\paramSupp$ into $M$ non-overlapping (and not necessarily uniform) subsets $\paramSupp_j$ for $j = 1, \ldots, M$, such that $\paramSupp = \bigcup_{j=1}^M \paramSupp_j$. 
To give an importance weighting to the bins, the following weights apply
\begin{figure}[tb]
	\centering
	\footnotesize
	\vspace{1.5ex}
	\includegraphics[width=\columnwidth]{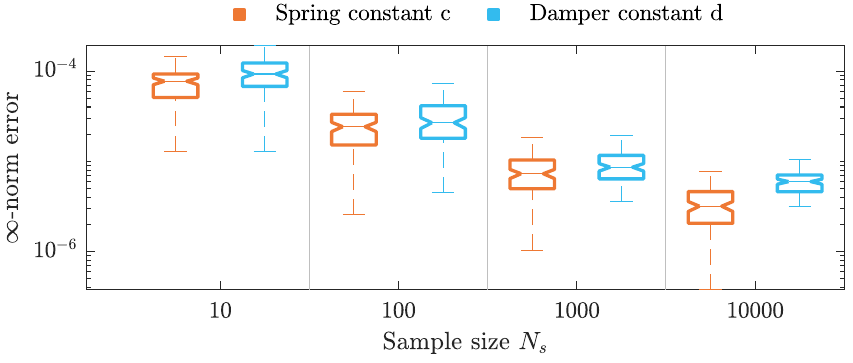}
\mycaption{Sampling error for different sample sizes}{Infinity norm error between the intrusive and non-intrusive method computed over $M=100$ experiments. The box plots illustrate the median, the 25th and 75th percentiles, while the whiskers denote the outliers.}
	\label{fig:SpringDamer:Norm_Error}
\end{figure}
\begin{align}
	p_j = | \paramSupp_j | / {N_\s} \quad j = 1, \ldots, M,
\end{align}
where $\sum_{j = 1}^M | \paramSupp_j | = {N_\s}$, see Figure~\ref{fig:scatter_plot} for an illustration. 
The bin-wise SSA of \eqref{eq:OT:Wasserstein_sensitivity} is then given by
{
\begin{align}
	 &\hat{\xi}_{W_\mathrm{B}^2}(\param,y) = \E_{\hat{\param}} \left( W^2_\mathrm{B}\left(\dist{y}^N, \dist{ y|\param }^N \right) \right) \nonumber\\
	&\approx \sum_{j=1}^{M} p_j \left( \mathscr{M}^2 \left( \hat{\mu}_{y}, \hat{\mu}_{y | \paramSupp_j} \right) + \mathscr{V}^2\left(\hat{\Sigma}_y, \hat{\Sigma}_{\dist{y| \paramSupp_j}} \right)\right), \label{eq:SSA}
\end{align}
}%
which, after normalization with the sample variance of $y$, results in an empirical approximation of the type-2 Wasserstein sensitivity index \eqref{eq:OT:normalized_sensitivity_wasserstein}, i.e.,
\begin{align}
	\hat{\iota}_\mathrm{B}(\param,y) = (2 \mathrm{tr}(\hat{\Sigma}_{y}))^{-1}\hat{\xi}_{W^2}(\param,y).
\end{align}
The first order Sobol' index can similarly be approximated as
\begin{align}
\label{eq:OT:normalized_approximated_sobol}
\hat{\iota}_S(\param,y) &= (\mathrm{tr}(\hat{\Sigma}_{y}))^{-1} \sum_{j=1}^{M} p_j \mathscr{M}^2 \left( \hat{\mu}_{y}, \hat{\mu}_{y | \paramSupp_j} \right).
\end{align}
\begin{remark}
	The intuition behind \eqref{eq:SSA} is the following: For each bin $\paramSupp_j$, the distance between the conditional distribution $\dist{ y|\paramSupp_j }^N$ and the marginal distribution $\dist{y}^N$ measures discrepancies in the output $y$ over the partitioned parameter space $\paramSupp$, see also Fig.~\ref{fig:scatter_plot}.
\end{remark}

\subsection{Example transport sensitivities}
\revv{The following revisits the spring-damper example introduced in Section~\ref{sec:intrusive:sens_example}.}
Unlike the intrusive method, the non-intrusive method relies on samples of the state trajectory to derive the Wasserstein sensitivities. 
To quantify the finite sample error between both methods, the analysis computes the maximum deviation between \eqref{eq:Sobol_first_order} and \eqref{eq:OT:normalized_approximated_sobol} along the trajectory $y(t)$ for different sample sizes $N_\s \in [10, 10^2, 10^3, 10^4]$, i.e.,
\begin{align*}
	e_i(j, N_{\mathrm{s}}) \coloneqq \max_t \left|\hat{\iota}_S(\param_j,y(t)) - \mathrm{SU}_{j}(y(t)) \right|, \quad i = 1, \ldots, M,
\end{align*}
and repeats the experiment for $M = 100$ times to enable statistical inferences. The results are visualized in Figure~\ref{fig:SpringDamer:Norm_Error}, illustrating the finite sample error between the intrusive and non-intrusive method for both the spring and damper constants.
As expected, increasing the sample size $N_\mathrm{s}$ improves the reliability of the Wasserstein sensitivity estimate. 
\revv{Consequently, the estimation error converges to zero as $N_\mathrm{s} \rightarrow \infty$, an observation that is attributable to the SAA in \eqref{eq:SSA}, see \cite[Thm. 21]{borgonovo2024global}.}
\subsection{Total-order optimal transport sensitivities}
Let $\param = (\param_1, \param_2, \ldots, \param_{q})$ be a $q$-dimensional random vector defined on the joint distribution $\dist{\param}$.
The conditional probability distribution of the output $y$ given $\param$ without $\param_i$ is denoted as $\dist{y|\sim \param_i}$.
To analyze the total-effect sensitivity of $y$ with respect to the parameter $\param_i$, the variance from all interactions of any order with any other parameter $\param_j$ for $j \neq i$ is taken into account. 
The idea was initially proposed by \cite{homma1996importance} using Euclidean distances, but is now extended using the optimal transport discrepancy \eqref{eq:OT:discrepancy}, i.e.
\begin{align}
	\xi^{\mathrm{T}}(\param_i, y) \coloneqq 1 - \mathbb{E}_{\mathbb{P}_{\sim \param_i}} \left( D_{\kappa}\left(\dist{y}, \dist{ y| \sim \param_i }\right) \right). \label{eq:OT:Wasserstein_sens_index}
\end{align}
Similar to Section \ref{subsec:first_order_wasser}, the type-2 Wasserstein distance follows by setting the transport cost to $\kappa(y_1, y_2) = |y_1 - y_2|^2$. Then, following the lines of \eqref{eq:OT:Wasserstein_sensitivity}, the total-order type-2 Wasserstein sensitivity index is defined as
\begin{align}
	\iota^{\mathrm{T}}_{\mathrm{B}}(\theta_i, y) \coloneqq 1 - \left(2\mathrm{tr}\left(\Sigma_{\dist{y}}\right) \right)^{-1} \xi^{\mathrm{T}}_{W^2_\mathrm{B}}(\param_i, y). \label{eq:OT:normalized_total_sensitivity_wasserstein}
\end{align} 
\begin{figure}[t]
	\centering
	\footnotesize
	\includegraphics{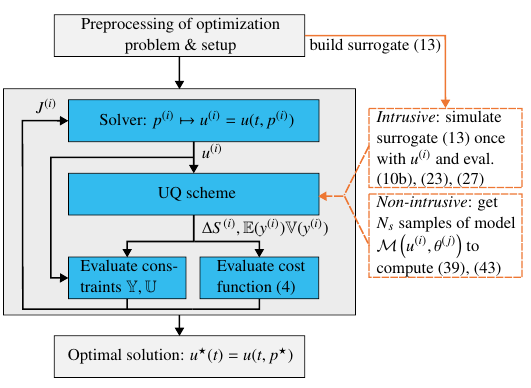}
	\caption{\raggedright Steps to solve the optimization problem~\eqref{eq:SOCP_discretized}.}
	\label{fig:algorithm}
\end{figure}
\begin{remark}
	Similar to Remark \ref{remark:type2_first_order}, the total-order type-2 Wasserstein sensitivity index can be related to its variance-based counterpart. 
	To this end, substitution of \eqref{eq:OT:Bures} into \eqref{eq:OT:normalized_total_sensitivity_wasserstein} gives
	\begin{align*}
		\iota^{\mathrm{T}}_\mathrm{B}(\param_i, y) &= 1 - \frac{\mathscr{M}^2( \mu_{\dist{y}}, \mu_{\dist{y|\sim \param_i}}) + \mathscr{V}^2(\mu_{\dist{y}}, \mu_{\dist{y|\sim \param_i}})}{2 \mathrm{tr}(\Sigma_{\dist{y}})},
	\end{align*}
	where the advective part $\mathscr{M}^2( \mu_{\dist{y}}, \mu_{\dist{y|\sim \param}})$ is connected to the total order Sobol' index as
	\begin{align}
		\iota^{\mathrm{T}}_S(\param, y) &= 1-\frac{\mathscr{M}^2( \mu_{\dist{y}}, \mu_{\dist{y|\sim \param_i}})}{\mathrm{tr}(\Sigma_{\dist{y}})}. \label{eq:OT:sobol_total_order}
	\end{align}
\end{remark}
\begin{remark}
	The total order sensitivity index \eqref{eq:OT:sobol_total_order} can be computed analogously to the first order index by gridding the parameter space. However, the total-order index requires a $q$-dimensional grid, whereas the first order index requires a $1$-dimensional grid. Therefore, the computation time to evaluate the total-order index increases exponentially in the parameter dimension $q$. 
\end{remark}
\section{The optimization method}
\label{sec:optimization_method}

\revv{While the two previous sections focused on the computation of sensitivities for a given input signal, the focus now returns to the optimization of the input signal itself.}
\revv{As stated in Section~\ref{ssec:OptimalExcitationProblem}, the optimal excitation is considered as the solution of the infinite-dimension stochastic optimal control problem~\eqref{eq:SOCP_input}.}
With the results from Sections~\ref{sec:intrusive_sensi} and~\ref{sec:nonintrusive_sensi}, the cost functional~\eqref{eq:cost_functional} can now be evaluated for a given input signal. 

\revv{However, problem~\eqref{eq:SOCP_input} is still challenging since the optimization is performed over an infinite-dimensional space of input functions.} \revv{This aspect is addressed in Section~\ref{ssec:Input}.} \revv{Afterwards, Section~\ref{ssec:OptProblem} discusses suitable optimization algorithms to solve the resulting finite-dimensional optimization problem, and Section~\ref{sec:IllustrativeExampleOID} shows exemplary results for the spring-damper system.}
\subsection{Parameterization of the input}
\label{ssec:Input}
\begin{figure*}
	\centering
	\footnotesize
	\vspace{2.0ex}
	\includegraphics[width=\linewidth]{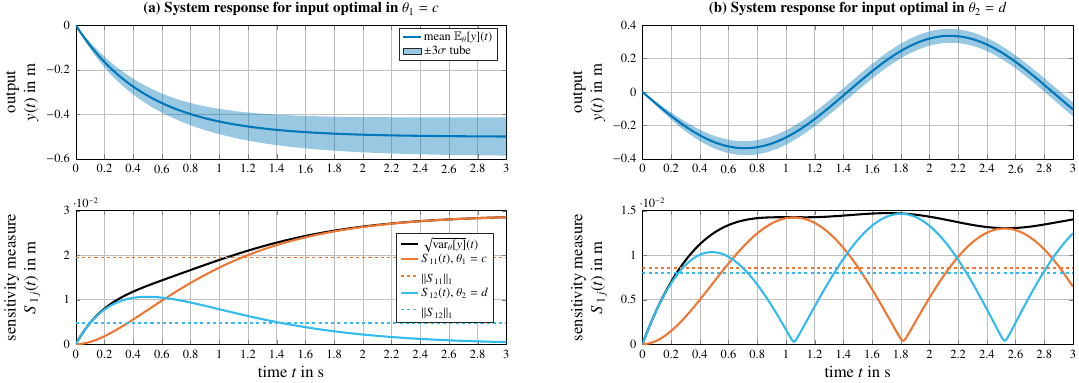}
	\vspace{-2.5ex}
	\mycaption{Spring damper system with optimal inputs}{\revv{Response of the spring-damper example under optimal inputs tailored for the identification of (a) the spring constant~$c$ and (b) the damper constant~$d$.} The mean output and the overall uncertainty are shown in the upper plots while the sensitivities of the parameters can be seen in the lower plots.}
	\label{fig:optInputSpringDamperSys}
\end{figure*}
Finding numerical solutions to optimal control problems typically involves the discretization of the input signal \cite{betts1998survey,liberzon2011calculus}. 
A common approach to parameterize an input signal~$u(\dof,t)$ with finite degrees of freedom~$\dof$ is using the superposition of basis functions $\rho_i(t,\dof)$. \revv{In general, this parameterization is nonlinear in~$\dof$ when the basis functions themselves depend on~$\dof$.} \revv{Fixed basis functions are a special case, where the parameterization reduces to a linear combination $u(t,\dof)=\dof^\T \rho(t)$.} For example Fourier methods, where the frequencies and amplitudes of sinusoidal basis functions are optimized, fall under this category and are widely used in optimal input design \cite{gevers2005optimal,ljung1999system}. 
Note that the direct discretization of the signal at certain time instances is a special case of this approach, where the basis functions are linear interpolants, splines or piecewise-constant functions~\cite{betts1998survey}.
An important requirement for the choice of basis functions is the ability to verify input constraints $u(t,\dof)\in\mathbb{U}$, e.g., amplitude or rate limits, for all times $t\in[0,t_\f]$ -- not solely at the certain points in time. Here the use of Bernstein polynomials is advantageous, since the constraints can be verified by checking the coefficients only \cite{farouki2012bernstein}.

For the sake of generality, the remainder retains the general notation $u(t,\dof)$, leading to the optimization problem for the degrees of freedom~$\dof\in\R^N$
\begin{align}
\label{eq:SOCP_discretized}
 &\max_{\dof\in\R^N}&  &J\big(\Delta S(t),u(t,\dof))&\qquad&\\
 &~\mathrm{s.\,t.}&  & y(t) = \model(u(t,\dof),\param),&\text{~for all } t\in[0,t_\f],\nonumber\\
 &&                 & y(t) \in \mathbb{Y}, ~~u(t,\dof) \in \mathbb{U},&\text{~for all } t\in[0,t_\f]\hphantom{,} \nonumber
\end{align}
\
being a finite-dimensional approximation of the original problem~\eqref{eq:SOCP_input}.
\subsection{Optimization algorithm}
\label{ssec:OptProblem}

To compute a solution of the optimization problem~\eqref{eq:SOCP_discretized}, suitable optimization algorithms are required. Since problem~\eqref{eq:SOCP_discretized} is typically non-convex and possibly non-smooth, gradient-free optimization algorithms such as evolutionary strategies~\cite{storn1997differential} or Bayesian optimization~\cite{snoek2012practical} are reasonable choices. 
For the examples presented in this work, the \texttt{scipy.optimize} implementation of the differential evolution algorithm~\cite{storn1997differential} from the SciPy library~\cite{2020SciPy-NMeth} solves~\eqref{eq:SOCP_discretized}.

The overall steps to solve the optimization problem~\eqref{eq:SOCP_discretized} are summarized in Fig.~\ref{fig:algorithm}. 
Once the objective, constraints, and uncertain parameters have been defined, the input signal is then parameterized in terms of a set of basis functions.
If the intrusive method is used to compute the sensitivities, the surrogate model is also derived in this first step.
Then, the optimization algorithm iteratively proposes new input parameters~$\dof^{(i)}$, which are used to evaluate the input signal~$u(t,\dof^{(i)})$. With the current input signal, the sensitivity measures and statistics of the outputs are computed using either the intrusive or non-intrusive method. The latter case requires a sampling algorithm and more function evaluations but less preprocessing effort. In the intrusive case, the computation of the sensitivity solely requires a single evaluation of the surrogate model with the current input~$u(t,\dof^{(i)})$.
Based on this information, the cost and the constraints of the optimization problem are evaluated and given back to the optimization algorithm. This process is repeated until a convergence criterion is met.
\subsection{Illustrative example: optimal input design}
\label{sec:IllustrativeExampleOID}
In the previous sections, two different methods were inspected to compute the sensitivities for the spring-damper system~\eqref{eq:spring_damper_example} with a fixed input signal.
The subsequent paragraphs demonstrate the use of the optimization method to find optimal sinusoidal input signals for the individual uncertain parameters~$c$ and~$d$.
To this end, \revv{the same model parameters as stated in Section~\ref{sec:intrusive:sens_example} apply} and the input set is defined as
\begin{align}
	\label{eq:input_set_spring_damper}
\mathbb{U}_1 =\{ &u(t) = u_0 \sin(2\pi f t - \varphi)~|~u_0 \in [0.0\,\mathrm{N},\,1.0\,\mathrm{N}],\nonumber\\
&f \in [0.0\,\mathrm{Hz},\,5.0\,\mathrm{Hz}],\, \varphi \in [0,2\pi] \},
\end{align}
which comes with the degrees of freedom $\dof^\T=[u_0,\,f,\,\varphi]$
while the output remains unconstrained. For the cost functional \eqref{eq:cost_functional}, the two pairs of weighting matrices read
\begin{table}[b]
	\caption{\raggedright Optimal parameters of the input signal given by~\eqref{eq:input_set_spring_damper}.}
	\vspace{-2.0ex}
	\label{tab:optimalsignals}
	\begin{center}	
	\begin{tabular}{cccc}
	\hline
	optimization task & $u_0$ in $\mathrm{N}$ & $f$ in $\mathrm{Hz}$ & $\varphi$  \\\hline
	(a):\hspace{2ex} $\param_1 = c$ & 1.000 & 0.002 & 1.590\\
	(b):\hspace{2ex} $\param_2 = d$ & 1.000 & 0.350 & 2.310\\ 
	\hline
	\end{tabular}
	\end{center}
\end{table}
\begin{align*}
Q_\mathrm{a}=\begin{bmatrix}
	1 & 0 \\ 0 & 0 \\
	\end{bmatrix},~~~
	~~Q_\mathrm{b}=\begin{bmatrix}
	0 & 0\\ 0 & 1
	\end{bmatrix},~~~ R_\mathrm{a}=R_\mathrm{b}=0
\end{align*}
\begin{figure}[b]
\centering
\footnotesize
\vspace{-1ex}
\includegraphics{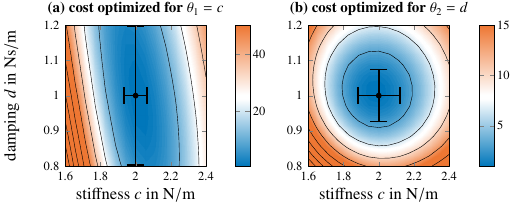}
\mycaption{Shape of the cost function for identification and resulting estimates}{Input optimized for (a) the spring constant~$c$ and (b) the damper constant~$d$. \rev{The 3$\sigma$-confidence intervals of the resulting parameter estimates $\hat c$ and  $\hat d$ are indicated by the black bars.}}
\label{fig:optCostSpringDamperSys}
\end{figure}
to maximize the sensitivity either w.r.t. the spring coefficient~$\theta_1=c$ or the damper constant~$\theta_2=d$, respectively. The optimization algorithm from Section~\ref{ssec:OptProblem} is then used to solve the two optimization problems~\eqref{eq:SOCP_discretized}. To this end the differential evolution algorithm is configured with a population size of~$N_ \mathrm{pop}=20$ and terminated after 24~(task a) and 28~(task b) iterations. A gradient-based optimization is done after the differential evolution to refine the solution.

With both methods for sensitivity analysis presented in Sections~\ref{sec:intrusive_sensi} and~\ref{sec:nonintrusive_sensi}, the same results follow but with different computation times (i.e. function evaluations). 
Since the algorithm uses $N_\mathrm{pop}\cdot\mathrm{dim}(\dof)=60$ function evaluations per iteration, the non-intrusive method with sample size $N_\s=100 $ needs $N_\s N_\mathrm{pop}\cdot\mathrm{dim}(\dof)=6\cdot 10^3$ simulations of the original model~\eqref{eq:spring_damper_example}.
Contrarily the intrusive procedure requires solely a single simulation of the surrogate per evaluation of the cost function. This corresponds to $N_\mathrm{pop}\cdot\mathrm{dim}(\dof)=60$~\revv{per iteration}.

Tab.~\ref{tab:optimalsignals} lists the resulting optimal input signals. The resulting outputs for both optimization tasks are shown in Fig.~\ref{fig:optInputSpringDamperSys} (upper graphs) along with corresponding sensitivity time series (lower graphs). 

The optimal input signal for the spring constant~$c$ tends to a step with the maximum amplitude. \revv{This result is expected since the spring constant is very dominant in the steady-state solution $\bar x = \bar u / c$ of the system.} \revv{The lower plot in Fig.~\ref{fig:optInputSpringDamperSys}(a) confirms that the damping has little influence in steady state.}
On the other hand, the optimal input signal for the damper constant~$d$ is shown in Fig.~\ref{fig:optInputSpringDamperSys}(b). Here, the optimal solution trades off the increasing damping for higher frequencies and the higher impact of the damper in transient states. However, the optimal sensitivities for the damping are significantly smaller compared to the spring constant case.

To illustrate the benefits in the parameter identification process using the optimized input signals, this illustration synthetically generates experimental data with a known ground truth of the parameters, $c^\star = 2.0\,\mathrm{N/m}$ and $d^\star= 1\,\,\mathrm{Ns/m}$. 
The experimental data is generated synthetically by adding measurement noise to the predicted data.
The resulting shape of the cost function is shown in Fig.~\ref{fig:optCostSpringDamperSys} for both optimization tasks.
\revv{In accordance with the optimal sensitivities, the cost optimized for the spring constant~$c$ shows a significantly steeper descent towards the ground truth compared to the cost optimized for the damper constant~$d$.}

\rev{\revv{The benefit of the optimal excitation is reflected in the confidence intervals of the resulting parameter estimates, which are shown in black in Fig.~\ref{fig:optCostSpringDamperSys}.} \revv{Using an unbounded least-squares estimate~\cite{ljung1999system} for both parameters simultaneously, the mean values are consistent with the ground truth for both optimization tasks with relative deviations of $0.13\%$ for task~(a) and $0.16\%$ for task~(b). As residual, the squared error between the predicted output and the synthetically generated data with a sample time of~$0.05\,\mathrm{s}$ is used. The standard deviations and the resulting confidence intervals are computed from a Bayesian posterior approximation~\cite{bishop2006} and read as}}
\rev{
\begin{align*}
&\mathrm{task~(a):}&  \sigma_\mathrm{a}(\hat c ) &= 0.0220\,\mathrm{N/m},& \sigma_\mathrm{a}(\hat d) &= 0.0655\,\mathrm{Ns/m}& \\
&\mathrm{task~(b):}&  \sigma_\mathrm{b}(\hat c) &= 0.0397\,\mathrm{N/m},& \sigma_\mathrm{b}	(\hat d) &= 0.0246\,\mathrm{Ns/m}.&
\end{align*}
}
\rev{\revv{As the expected information of the damping constant~$d$ is very low for task (a), the standard deviation of the parameter estimate remains high.} \revv{This observation is in line with the optimal sensitivities shown in Fig.~\ref{fig:optInputSpringDamperSys}(a) and the shape of the cost function in Fig.~\ref{fig:optCostSpringDamperSys}.} \revv{Combining the data from both individual optimization tasks significantly reduces the uncertainty of both parameter estimates, i.e. $\sigma_{\mathrm{a}\cup \mathrm{b}}(\hat c) = 0.0225\,\mathrm{N/m}$ and $\sigma_{\mathrm{a}\cup \mathrm{b}}(\hat d) = 0.0266\,\mathrm{Ns/m}$.}
}
\section{Industry application: optimal input design for\\ vehicle systems}
\label{sec:IndustryExample}
As an industry example, this section considers a vehicle dynamics model commonly used in automotive applications for control and simulation.
The model captures the vehicle's lateral dynamics using a single-track representation, which simplifies its behavior yet provides a sufficient approximation for a wide range of driving scenarios.
The setting is visualized in Figure~\ref{fig:IllustrationSTM} showing the most important vehicle states.
Single-track models form the basis for various automotive motion control applications such as lane-keeping assistance~\cite{falcone2007predictive}, electronic stability control~\cite{van2000bosch}, or automated driving systems~\cite{snider2009automatic}.
The objective is to determine an optimal steering angle trajectory~$u$ that maximizes the identifiability of the unknown parameters $\theta$ from the measured output~$y$.
In this example, the yaw rate serves as output while the inertia around the vertical axis shall be identified.

Within this chapter, the vehicle model is first introduced in Section~\ref{ssec:VehicleModeling}, and afterwards the optimization method is applied to find optimal steering trajectories in Section~\ref{ssec:OIDVehicle}. 
\rev{Finally, the ultimate benefit is illustrated for a real-world by comparing the parameter estimates obtained with the optimal input against standard input signals in Section~\ref{ssec:ResultsVehicle} in a real-world setting.}
\subsection{Vehicle System Modeling}
\label{ssec:VehicleModeling}
A nonlinear model serves to compare the intrusive approach from Section~\ref{sec:intrusive_sensi} against the non-intrusive approach from Section~\ref{sec:nonintrusive_sensi}. For the latter, the model is \revv{compiled} as a black box, \revv{and further used without applying any knowledge about its underlying mathematical structure.}
Afterwards, a linearized version of the model serves to compute an intrusive surrogate model.\\[1ex]
\noindent\textit{Nonlinear vehicle model}. The dynamics of the yaw rate~$\dot{\psi}$, the side slip angle~$\beta$, and the longitudinal speed~$v$ are given by
\begin{subequations}
	\label{eq:vehicle_dynamics_nonlinear}
\begin{align}
\ddot{\psi} &= \frac{1}{J_\mathrm{z}} \left( l_\mathrm{f} F_\mathrm{y,f} \cos\,\delta - l_\mathrm{r} F_\mathrm{y,r} \right)\\
\dot{\beta} &= \frac{1}{m v} \left( F_\mathrm{y,f} \cos\,\delta + F_\mathrm{y,r} \right) - \dot{\psi} \\
\dot{v} &= \frac{1}{m} \left( F_\mathrm{x} \cos\,\beta + F_\mathrm{y,f} \sin(\delta - \beta) + F_{y,r} \sin\,\beta \right),
\end{align}
\end{subequations}
where~$m$ is the vehicle mass, $J_\mathrm{z}$ is the yaw moment of inertia,~$l_\mathrm{f}$ and~$l_\mathrm{r}$ are the distances from the center of gravity to the front axle (index $\mathrm{f}$) and the rear one (index $\mathrm{r}$). Since the focus of this investigation is on the lateral dynamics, no external accelerating force is assumed, i.e.~$F_\mathrm{x} = 0$. 
The actual lateral tire forces~$F_\mathrm{y,f}$ and~$F_\mathrm{y,r}$ tend to their stationary levels by first-order dynamics. The latter ones are given by Pacejka's simplified tire model~\cite{Pajecka1987tire} 
\begin{subequations}
	\label{eq:tire_forces}
\begin{align}
	T_\mathrm{f}\dot{F}_\mathrm{y,f}=F_\mathrm{y,f}-\frac{\mu_{\mathrm{f}}\, C_{\mathrm{f}}\, m\,g\,l_{\mathrm{r}}}{l_{\mathrm{f}} + l_{\mathrm{r}}}\, \arctan(B_{\mathrm{f}}\, \alpha_{\mathrm{f}})\\
	T_\mathrm{r}\dot{F}_\mathrm{y,r}=F_\mathrm{y,r}-\frac{\mu_{\mathrm{r}}\, C_{\mathrm{r}}\, m\,g\,l_{\mathrm{f}}}{l_{\mathrm{f}} + l_{\mathrm{r}}}\, \arctan(B_{\mathrm{r}}\, \alpha_{\mathrm{r}}),
\end{align}
\end{subequations}
both being functions of the slip angles $\alpha_\mathrm{f}$ and $\alpha_\mathrm{r}$, respectively. These are defined as the difference between the steering angle $\delta$ and the actual slip angle of the front and rear tires and given by
\begin{subequations}
	\label{eq:slip_angles}
\begin{align}
\alpha_{\mathrm{f}} &= \delta - \arctan\left( \frac{l_{\mathrm{f}} \dot{\psi} + v \sin\beta}{v \cos\beta} \right) \\
\alpha_{\mathrm{r}} &= -\arctan\left( \frac{l_{\mathrm{r}} \dot{\psi} - v \sin\beta}{v \cos\beta} \right).
\end{align}
\end{subequations}
In order to acccount for the fact, that a desired steering angle~$u$ with $\nu=1$ is not immediately applied to the vehicle, the second-order lag element
\begin{figure}[t]
	\centering
	\footnotesize
	\includegraphics[width=\linewidth]{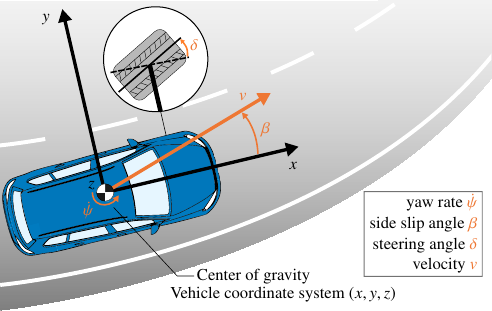}
	\caption{\raggedright Illustration of the lateral vehicle dynamics.}
	\label{fig:IllustrationSTM}
\end{figure}
\begin{align}
\label{eq:steering_dynamics}
\ddot{\delta} + 2 d_\mathrm{s} \omega_\mathrm{s} \dot \delta + \omega_\mathrm{s}^2 \delta = \omega_\mathrm{s}^2 u
\end{align}
with uncertain natural frequency~$\omega_\mathrm{s}$ and damping ratio~$d_\mathrm{s}$ is used to model the steering dynamics. 
In total, the nonlinear vehicle model is given by~\eqref{eq:vehicle_dynamics_nonlinear}--\eqref{eq:steering_dynamics}, with the state $x=[\dot\psi,\,\beta,\,v,\,{F}_\mathrm{y,f},\,{F}_\mathrm{y,r},\delta,\,\dot\delta]$. 
In this example, the yaw rate $\dot\psi$ is considered as scalar output~$y$, i.e. $\nu=1$.
The initial conditions are set to zero, except for the longitudinal speed, which is initialized with $v_0 = 13.89\,\mathrm{m/s}$. 
The remaining vehicle parameters are assumed to be deterministic and listed in Table~\ref{tab:vehicle_parameters}.

The prior knowledge of the inertia is reflected by the normal distribution $J_\mathrm{z} \sim \N(6000\,\mathrm{kg\,m}^2,1000\,\mathrm{kg\,m}^2)$.
As additional uncertain parameters, the tire stiffnesses~$B_\mathrm{f} \sim\N(10,\,1.0)$ and  $B_\mathrm{r} \sim\N(10,\,1.0)$ from~\eqref{eq:tire_forces} enter the model, as well as the parameters of the steering dynamics~\eqref{eq:steering_dynamics}, i.e.~$\omega_\mathrm{s} \sim\N(17\,\mathrm{\tfrac{rad}{s}},\,4.0\,\mathrm{\tfrac{rad}{s}})$ and the damping ratio~$d_\mathrm{s} \sim\N(0.75,\,0.05)$. In summary, a total of $\np=5$ uncertain parameters $\param = [J_\mathrm{z}, B_\mathrm{f}, B_\mathrm{r}, \omega_\mathrm{s}, d_\mathrm{s}]^\top$ result.\\[1ex]
\noindent\textit{Linear single track model}. 
Several typical assumptions yield a linearized single track model. These include small angle approximations, neglecting the tire dynamics~\eqref{eq:tire_forces}, and assuming the lateral tire forces~${F}_\mathrm{y,(\cdot)}$ are affine functions of their respective slip angles~$\alpha_{(\cdot)}$.
The nonlinear relations~\eqref{eq:vehicle_dynamics_nonlinear}--\eqref{eq:slip_angles} simplify to the linear ODEs
\begin{subequations}
	\label{eq:vehicle_dynamics_linear}
\begin{align}
\ddot{\psi} &= \frac{-c_\mathrm{f} l_\mathrm{f}^2 - c_\mathrm{r} l_\mathrm{r}^2}{J_z} \dot{\psi}
+ \frac{ c_\mathrm{r} l_\mathrm{r}-c_\mathrm{f} l_\mathrm{f} }{J_z} \beta + \frac{c_\mathrm{f} l_\mathrm{f}}{J_z}  \, \delta \\
\dot{\beta} &= \left(\frac{-c_\mathrm{f} l_\mathrm{f} + c_\mathrm{r} l_\mathrm{r} }{m v^2} -1 \right) \dot{\psi}
- \frac{c_\mathrm{f} + c_\mathrm{r} }{m v} \beta
+ \frac{c_\mathrm{f}}{m v}  \, \delta.
\end{align}
\end{subequations}
In addition, the vehicle speed~$v=v_0$ becomes a constant parameter instead of a state. 
To consider the neglected dynamics, the speed is also treated as an uncertain parameter with $v \sim\N(11.27\,\mathrm{m/s},\,0.87\,\mathrm{m/s})$.
The cornering stiffnesses~$c_{(\cdot)}$ at the respective axle are given by
\begin{subequations}
	\label{eq:cornering_stiffnesses}
\begin{align}
c_\mathrm{f} &= K_\mathrm{f}  \frac{ m g l_\mathrm{r}}{l_\mathrm{f} + l_\mathrm{r}},&&K_\mathrm{f} \sim\N\left(9.53\,\mathrm{\frac{s}{m}},\, 1.20\,\mathrm{\frac{s}{m}}\right)& \\
c_\mathrm{r} &= K_\mathrm{r} \frac{ m g l_\mathrm{f}}{l_\mathrm{f} + l_\mathrm{r}},&&K_\mathrm{r} \sim\N\left(18.8\,\mathrm{\frac{s}{m}},\, 2.00\,\mathrm{\frac{s}{m}}\right).&
\end{align}
\end{subequations}
To consider the neglected nonlinearities, the standard deviation of the combined stiffnesses $K_{(\cdot)} = \mu_{(\cdot)} C_{(\cdot)} B_{(\cdot)}$ are overapproximated as stated above. 
This results in $\np = 6$ uncertain parameters~$\param = [J_\mathrm{z}, K_\mathrm{f}, K_\mathrm{r}, \omega_\mathrm{s}, d_\mathrm{s}, v]^\top$ and a $\nx = 4$ dimensional state vector $x=[\dot\psi,\,\beta,\,\delta,\,\dot\delta]^\top$ for the linearized model. 
In accordance with the nonlinear model, the initial condition vanishes.
In addition, it is straightforward to verify that the linear model featuring the ODEs~\eqref{eq:vehicle_dynamics_linear} and~\eqref{eq:steering_dynamics} complies with the LPV-model structure defined in~\eqref{eq:LPV_stochastic} for the intrusive method from Section~\ref{sec:intrusive_sensi}.
\begin{table}[b]
	\caption{\raggedright Vehicle Parameters of the exemplary vehicle model representing a middle class car.}
	\label{tab:vehicle_parameters}
	\centering
	\begin{tabular}{llll}
		\hline
		\textbf{Parameter description} & \textbf{Symbol} & \textbf{Unit} & \textbf{Value} \\
		\hline
		Initial vehicle speed & $v$ & $\mathrm{m}/\mathrm{s}$ & 13.89 \\
		Vehicle mass & $m$ & $10^3\,\mathrm{kg}$ & 2.700 \\
		Distance CoG to front axle & $l_\mathrm{f}$ & $\mathrm{m}$ & 1.548 \\
		Distance CoG to rear axle & $l_\mathrm{r}$ & $\mathrm{m}$ & 1.441 \\
		Front tire friction coefficient & $\mu_\mathrm{f}$ & -- & 1.000 \\
		Rear tire friction coefficient & $\mu_\mathrm{r}$ & -- & 1.000 \\
		Front tire stiffness & $C_\mathrm{f}$ & $\mathrm{N}/\mathrm{rad}$ & 0.953 \\
		Rear tire stiffness & $C_\mathrm{r}$ & $\mathrm{N}/\mathrm{rad}$ & 1.878 \\
		Gravity constant & $g$ & $\mathrm{m}/\mathrm{s}^2$ & 9.810 \\
		Front tire time constant & $T_\mathrm{f}$ & $\mathrm{ms}$ & 28.57 \\
		Rear tire time constant & $T_\mathrm{r}$ & $\mathrm{ms}$ & 28.57 \\
		\hline
	\end{tabular}
\end{table}
\subsection{Optimization of the vehicle's excitation}
\label{ssec:OIDVehicle}
The objective is to find an optimal steering angle trajectory~$u$ that maximizes the sensitivity of the yaw rate~$\dot\psi$ with respect to the inertia~$J_\mathrm{z}$. In comparison to the sinusiodal input space from Section~\ref{sec:IllustrativeExampleOID}, this example considers more advanced input signals that comprise the superposition of multiple ramp signals. An individual ramp is parameterized with $\dof = [u_0,~u_T,~t_T,~t_\Delta]^\T$ and is defined as
\begin{align}
\label{eq:ramp_input}
\rho(t,p)=\mathrm{sat}_{[u_0,\,u_T]}\left(u_0 + \tfrac{t - (t_T - t_\Delta)}{t_\Delta} (u_T - u_0)\right),
\end{align}
where $u_0$ and $u_T$ denote the start and end values, $t_T$ the final time and $t_\Delta\leq t_T$ the rise time.
In addition, the saturation function $\mathrm{sat}_{[a,b]}(x) = \max\{  \min(a,b),\, \min\{x, \max(a,b)\}\}$ is used in the equation above.

The space~$\mathbb{U}_2$ of possible excitation functions is then defined as the set of all possible superpositions of $N$ ramp signals, i.e.,
\begin{align}
	\label{eq:input_set_vehicle}
\mathbb{U}_2 &= \bigg\{ u(t) = \sum_{i=1}^{N} \rho(t,p^i)~\bigg|~u_0^i =0,~t^i_\Delta\leq t^i_T,
~t_T^i\in(0,T], \nonumber\\
& ~u(0)=u(T)=0, ~|u(t)| < u^+, |\dot{u}(t)| < \dot{u}^+, \forall t \in [0, T] \bigg\},
\end{align}
with the time horizon $T=10.0\,\mathrm{s}$, the maximum amplitude $u^+ = 0.14\,\mathrm{rad}$ and a rate limit of~$\dot{u}^+ = 0.157\,\mathrm{\tfrac{rad}{s}}$. 
The choice $N=4$ yields $p=[p^1, \ldots, p^N]^\T\in\R^{16}$ total degrees of freedom. 
Similar to the previous example, the output~$y$ remains unconstrained. 
In addition, process and measurement noise vanishes according to the model representation. 
This implies $S_{\min} (t)=0$ and leads to the effective sensitivities $\Delta S(t)=S(t)$, see~\eqref{eq:sensitivity_effective}.
The differential evolution algorithm is configured with a population size of $N_\mathrm{pop}=30$ and terminated after 250 iterations. 
For the non-intrusive approach, the sample size~$N_\s=50$ is used.\\[1ex] 
\noindent\textit{Setup and computationally complexity}. Due to the structure of the nonlinear model~\eqref{eq:vehicle_dynamics_nonlinear}--\eqref{eq:steering_dynamics}, the intrusive method with the presented LPV-structure is not applicable. 
Therefore, the optimization problem~\eqref{eq:SOCP_discretized} is first solved based on the nonlinear model using the non-intrusive method from Section~\ref{sec:nonintrusive_sensi}. 
Subsequently, the linearized model~\eqref{eq:vehicle_dynamics_linear}--\eqref{eq:steering_dynamics} is used to build a PCE-surrogate model for the intrusive method from Section~\ref{sec:intrusive_sensi}. Due to the amount of $\np=6$ parameters, building a surrogate model with PCE-order~$\npce=3$ results in $\ell\nx = 336$ states and $\ell\ny = 84$ outputs with the lexicographical truncation scheme. 
On a standard Intel(R) Xeon(R) 6346 CPU, the computation time is reasonably small with $1.88\,\mathrm{s}$ using a quadrature order of $N_q=4$ and a sparse Smolyak grid. 
An increase of~$\npce$ or~$N_q$ did not lead to different simulation results.
Still, the mean computation time for a single simulation of the surrogate of $\E[t^\mathrm{c}_\mathrm{lin}]=0.58\,\mathrm{s}$ is comparable to a a single simulation of the nonlinear model which is $\E[t^\mathrm{c}_\mathrm{nl}]=0.46\,\mathrm{s}$.

This implies a significant impact on the computation time for an iteration of the optimization algorithm. The following estimates result for the linear and nonlinear case
\begin{align*}
	\E[t^\mathrm{c,iter}_\mathrm{lin}] &= \E[t^\mathrm{c}_\mathrm{lin}]\hphantom{N_\s }N_\mathrm{pop} \dim{\dof}  = 4.64\,\mathrm{min}\\
	\E[t^\mathrm{c,iter}_\mathrm{nl}] &= \E[t^\mathrm{c}_\mathrm{nl\hphantom{i}}]N_\s N_\mathrm{pop} \dim{\dof}  = \hphantom{.}184\,\mathrm{min},
\end{align*}
which differs mainly due to the $480$ instead of $2.4\cdot 10^4$ function calls per iteration. Certainly, the computation time for the nonlinear non-intrusive case can be reduced by using parallelization strategies, which is not considered here.\\[1ex]
\begin{figure}[tb]
	\centering
	\footnotesize
	\vspace{1.0ex}
	\includegraphics[width=\linewidth]{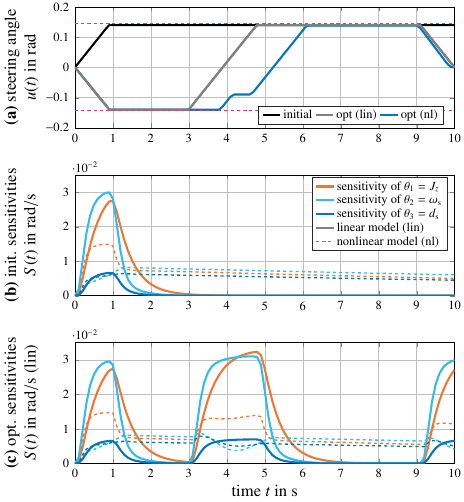}
	\mycaption{Optimal excitation signals for the vehicle systems}{The optimal inputs for the linear (grey) and nonlinear (dark blue) are shown in addition to an initial guess (black) and the amplitude restrictions (dashed lines) in (a). The other graphs show the resulting sensitivity trajectories for both models under the initial guess in (b) and the optimal input of the linear model in (c). The sensitivites according to the linear model are depicted in solid lines, the ones of the nonlinear model in dashed lines.}
	\vspace{-4ex}
	\label{fig:PCEvsOTOptimalSensitivities}
\end{figure}
\noindent\textit{Optimal input signals}. The resulting optimal input signals for the nonlinear and linear model are depicted in Fig.~\ref{fig:PCEvsOTOptimalSensitivities}(a). 
The ramps with maximum slope lead to the highest sensitivities. 
Due to the constraints of the input set~\eqref{eq:input_set_vehicle} the linear model would not benefit from using the fourth ramp segment due to the simplification. 
In contrast, the nonlinear model utilizes the full input space, applying all four ramps for system excitation.
Nonetheless, the two optimal input trajectories display comparable overall characteristics, leading to the conclusion that the linear model captures the relevant dynamics of the vehicle sufficiently accurate.\\[1ex]
\noindent\textit{Resulting sensitivity trajectories}. 
Figure~\ref{fig:PCEvsOTOptimalSensitivities}(b) illustrates the resulting sensitivity trajectories for both models, obtained by actuating each system with an initial input guess~$u^\mathrm{init}\notin\mathbb{U}_2$. This guess (depicted in Figure~\ref{fig:PCEvsOTOptimalSensitivities}(a)) incorporates a ramp segment with maximum slope, yet it fails to satisfy the terminal condition~$u(t_T)=0$.
For the sake of simplicity, the figure shows only the sensitivities w.r.t. the inertia~$J_\mathrm{z}$ (dark blue, the objective) and the two other common parameters (light blue/orange).
Again, both models exhibit similar sensitivity characteristics, however, due to the neglected dynamics, the linear model does not consider the sensitivity in steady-state.
In Figure~\ref{fig:PCEvsOTOptimalSensitivities}(c) the sensitivities for both models under the optimal input signals of the linear model are shown, i.e., the grey line in Figure~\ref{fig:PCEvsOTOptimalSensitivities}(a). Compared to the initial guess, both models exhibit significantly increased sensitivities for the inertia~$J_\mathrm{z}$ (dark blue). 
Despite being optimized for the inertia, it proved to be less dominant compared to other parameters in both simulations. 
A potential increase could be achieved with a less restrictive input set.\\[1ex]
\noindent\textit{Key findings.} The results indicate that the input design with the linearized model is able to find an optimal excitation that is comparable to the one obtained from the nonlinear model. 
Importantly, this aspect is an inherent system property, meaning it is independent of the UQ method chosen for the sensitivity analysis.
From a computational perspective, the intrusive method offers significant performance advantages over the non-intrusive approach, provided it is applicable to the system.
In this example, building the intrusive surrogate model requires negligible effort, even within an industrial application context.
While the intrusive method offers the major advantage of requiring no sampling for cost function evaluations, the non-intrusive method, conversely, provides the clear benefit of applicability to black-box models without further assumptions or simplifications.
\rev{\subsection{Parameter identification for the vehicle system}
\label{ssec:ResultsVehicle}
\revv{Finally, a test vehicle demonstrates the benefit of the method for system identification.}
\revv{The test vehicle is a modified VW ID.Buzz minibus with access to the steering control interface.}
The vehicle is shown in Fig.~\ref{fig:testvehicle}. 
\revv{The open steering control interface allows directly applying optimized control inputs to the vehicle.}
\revv{The parameter identification in this study is focused on the yaw inertia~$J_\mathrm{z}$ and on the stiffnesses~$K_\mathrm{f}$ and~$K_\mathrm{r}$, where the latter are lumped parameters that combine multiple Pacejka tire-model effects under the selected operating conditions. In addition, steering-dynamics parameters are calibrated separately at the actuator level and are therefore not included here.}}\\[1ex]
\begin{figure}[tb]
	\footnotesize
	\vspace{1.5ex}
	\includegraphics[width=0.48\textwidth]{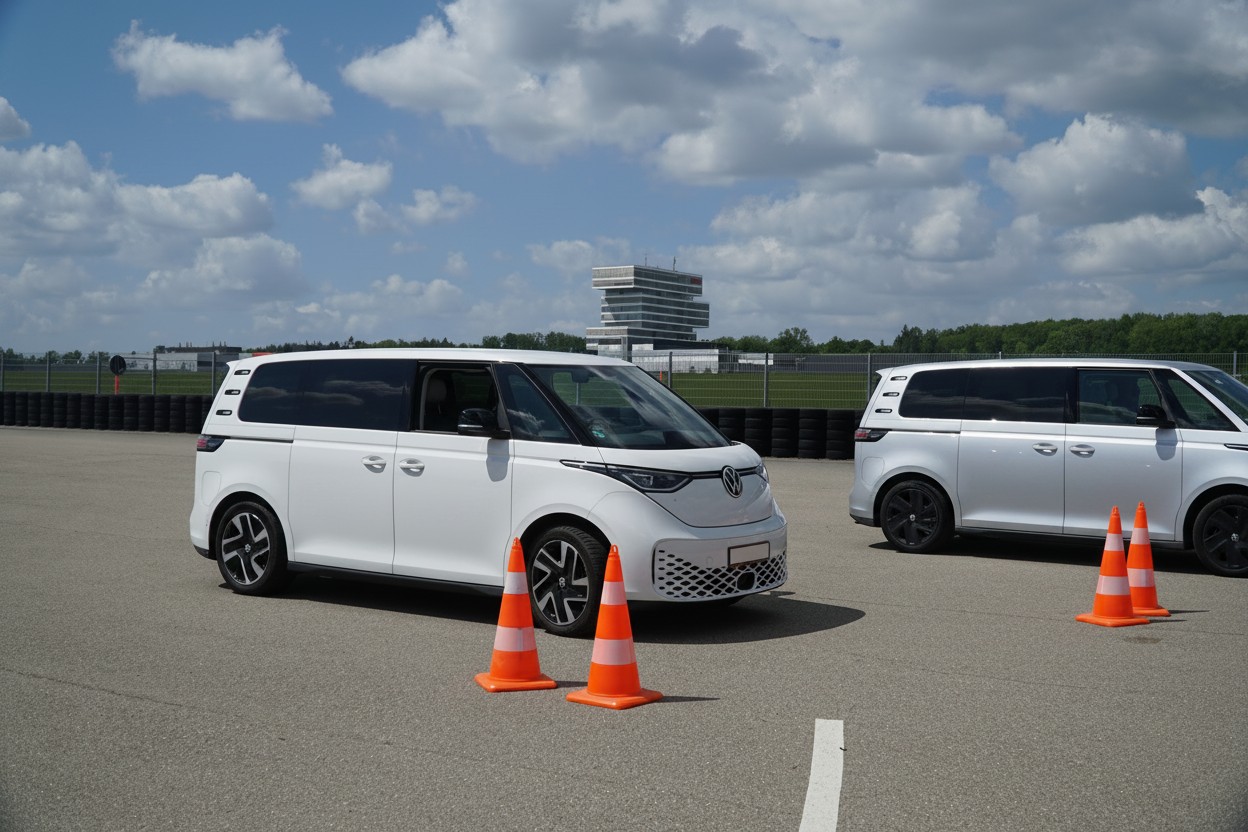}
	\caption{\raggedright \revv{Test vehicle ID.Buzz on a closed test track.}}
	\label{fig:testvehicle}
\end{figure}
\begin{table}[b]
	\mycaption{\rev{Identification results for different data sets}}{\rev{The ranges below correspond to the mean and the standard deviation of the identified parameters, i.e. $\E[\hat\theta_i] \pm \sigma(\hat\theta_i)$, for the given data set. Reference maneuvers (orange) and optimized ones (blue) are shown in Fig.~\ref{fig:measurement_overview}. The total time $T_\mathrm{tot}$ corresponds to the sum of the durations of the individual maneuvers.}}
	\label{tab:ident_results_vehicle}
	\centering
	\small
	\resizebox{\columnwidth}{!}{%
	\begin{tabular}{lcccc}
		\hline
		\textbf{Data Set} $\mathscr{D}$ & $T_\mathrm{tot}$ in s & $\hat J_\mathrm{z}$ in $\mathrm{kg\,m^2}$ & $\hat K_\mathrm{f}$ in $\mathrm{s/m}$ & $\hat K_\mathrm{r}$ in $\mathrm{s/m}$ \\
		\hline
		Ref. \#1--\#3 & 20.0 & $2397 \pm 1473$ & $6.245 \pm 0.512$ & $14.17 \pm 1.21$ \\
		Ref. \#1--\#5 & 40.0 & $2562 \pm \hphantom{0}858$ & $6.240 \pm 0.492$ & $14.18 \pm 1.20$ \\
		Opt. \#1--\#2 & 20.0 & $2542 \pm \hphantom{0}585$ & $6.291 \pm 0.481$ & $13.98 \pm 1.19$ \\
		\hline
	\end{tabular}%
	}
\end{table}
\begin{figure*}[!p]
	\centering
	\footnotesize
	\includegraphics{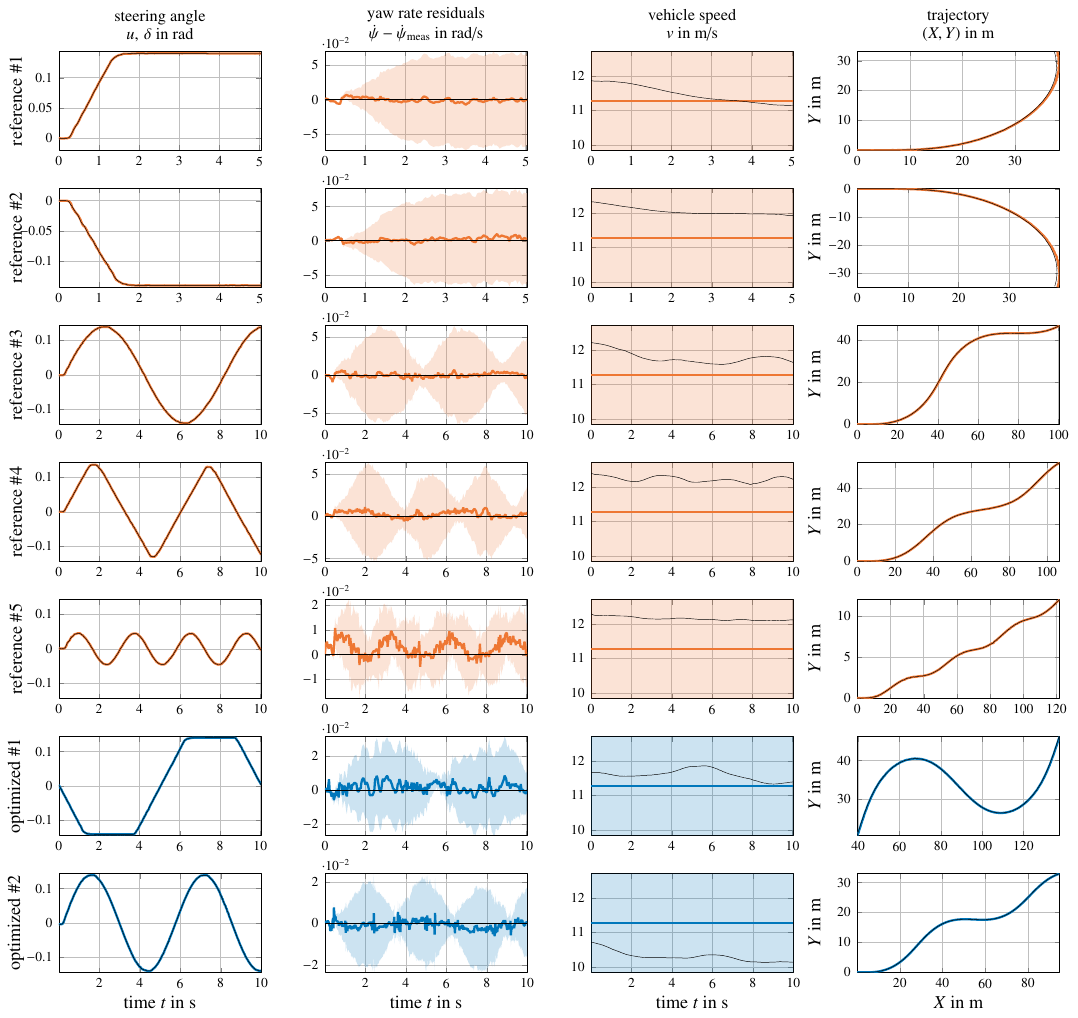}
	\vspace*{2ex}
	\mycaption{\rev{Identification results for a reference input set and two optimized inputs}}{\rev{Each row of this figure corresponds to one maneuver; the row label on the left indicates whether it belongs to the reference set~(Ref.~\#1--\#5, orange labels) or to the optimized set~(Opt.~\#1--\#2, blue labels). The columns show, from left to right, the desired and actual steering angles~$u$, $\delta$, the yaw-rate residual~$\dot\psi-\dot\psi_\mathrm{meas}$, the longitudinal speed~$v$, and the resulting vehicle trajectory in absolute coordinates. Within each graph, thin black solid curves denote the measured signals, thick coloured lines the expected values, and the shaded areas the $5\%$--$95\%$ confidence bounds.}}
	\label{fig:measurement_overview}
\end{figure*}
\rev{\noindent\textit{Considered maneuvers}.
To validate the proposed optimal excitation design, the comparison contrasts its resulting parameter estimates against those obtained from standard field-test maneuvers, with both methods applied to the same test vehicle.\\
The measurement results for all maneuvers are compiled in Fig.~\ref{fig:measurement_overview}, where each row corresponds to one maneuver. The columns show, from left to right, the applied steering angle~$\delta$, the yaw-rate residual~$\dot\psi - \dot\psi_\mathrm{meas}$ between the model prediction and the measurement, the longitudinal speed~$v$ (from which the covered speed range can be read directly), and the resulting vehicle trajectory in absolute coordinates. \revv{All considered maneuvers start from the same operating conditions.}
Further, note that the following analysis is based on the linear model~\eqref{eq:vehicle_dynamics_linear}--\eqref{eq:steering_dynamics}, which is used to identify the parameters based on the measured data. \revv{As a consequence, the speed~$v$ is treated as a static uncertain parameter in the linear model; its confidence set covers all performed measurements, and, as visible in Fig.~\ref{fig:measurement_overview}, the speed variations are slow compared to the lateral dynamics such that no relevant longitudinal-lateral coupling or significant bias is expected under the considered operating conditions.}\\
As a baseline, three standard automotive maneuvers are used: cornering left (Ref.~\#1), cornering right (Ref.~\#2), and a sinusoidal slalom (Ref.~\#3). \revv{The cornering maneuvers consist of a single turn with the maximum steering rate $\dot{u}^+$ to the maximum steering angle $u^+$ with a during of~$5.0\,\mathrm{s}$, while the $10.0\,\mathrm{s}$ sinusoidal slalom features with the maximum steering amplitude $u^+$ frequency of $0.13\,\mathrm{Hz}$.}
These represent common practice in vehicle testing and together form the first reference data set with a total duration of $T_\mathrm{tot} = 20\,\mathrm{s}$. \revv{They were selected as representative baseline handling maneuvers because they cover the standard left/right steady-cornering and oscillatory slalom excitations commonly used in vehicle-dynamics calibration and validation.} To obtain a more informative reference set, two additional maneuvers~\revv{are chosen before the optimization}: a ramp-like slalom (Ref.~\#4) and a high-frequency slalom with frequency $f=0.36\,\mathrm{Hz}$ (Ref.~\#5). All five reference maneuvers together constitute the second reference data set with $T_\mathrm{tot} = 40\,\mathrm{s}$. All reference maneuvers are depicted in orange in Fig.~\ref{fig:measurement_overview}. The steering amplitude and slope are chosen to match those of the optimal inputs from Section~\ref{ssec:OIDVehicle}, i.e. each maneuver uses either the maximum amplitude~$u^+$ or the maximum rate~$\dot{u}^+$.\\
Both optimized maneuvers are shown in blue in Fig.~\ref{fig:measurement_overview}. The ramp-based optimal input from Section~\ref{ssec:OIDVehicle} is used directly as the first optimized excitation (Opt.~\#1). To broaden the comparison, the optimization is additionally carried out for a sinusoidal input set based on the linear model with analogous constraints,}
\rev{
\begin{align*}
\mathbb{U}_3 =\{ u(t) = u_0 \sin(2\pi f t - \varphi)~|&~u(0)=0,~|u(t)| < u^+, \\ &|\dot{u}(t)| < \dot{u}^+, \forall t \in [0, T]  \},
\end{align*}}
\rev{yielding the optimal sinusoidal slalom maneuver denoted Opt.~\#2 \revv{with $u_0=1.39\,\mathrm{rad}$, $f=0.17\,\mathrm{Hz}$ and $\varphi=0$}.}\\[1ex]
\rev{\noindent\textit{Identification results}.
The identified parameters for the three data sets are summarized in Table~\ref{tab:ident_results_vehicle}. The ranges correspond to the mean and standard deviation of the identified parameters, i.e. $\E[\hat \theta_i] \pm \sigma(\hat \theta_i)$ for each data set.}

\rev{The parameter identification is performed using a unbounded least-squares approach~\cite{ljung1999system}, where the cost function is defined as the sum of squared yaw-rate residuals across all maneuvers in the \revv{respective data set~$\mathscr{D}$:}}
\revv{
\begin{align*}
J(\param) = \sum_{i\in\mathscr{D}} \sum_{k=1}^{k_{\max}^{(i)}} \frac{1}{t_{k_{\max}^{(i)}}}\left( \dot\psi^{(i)}(t_k; \param) - \dot\psi^{(i)}_\mathrm{meas}(t_k) \right)^2.
\end{align*}}
\revv{The measurements are sampled with a mean sample time of~$32.6\,\mathrm{ms}$. As can be seen above, each maneuver contribution is normalized by its duration. All maneuvers start from the same initial point. As initial point for the parameter estimation, the prior mean~$\E[\param]$ of the stochastic parameters is used.}

\revv{Analogous to the spring-damper example in Section~\ref{sec:IllustrativeExampleOID}, the reported means and standard deviations are obtained from a Bayesian posterior approximation~\cite{bishop2006}.} The resulting parameter estimates are then used to compute the confidence bounds shown as shaded areas in Fig.~\ref{fig:measurement_overview}, which represent the $5\%$--$95\%$ confidence intervals.
\rev{\noindent\textit{Discussion}.
The results in Table~\ref{tab:ident_results_vehicle} indicate that the cornering stiffness parameters $K_\mathrm{f}$ and $K_\mathrm{r}$ are consistently well excited across all maneuvers and therefore easier to identify, leading to comparable estimates in all three data sets. \revv{The remaining uncertainty level of the stiffnesses $K_\mathrm{f}$ and $K_\mathrm{r}$ is already acceptable for the targeted downstream simulation and control analyses.} Also the expected values of all three parameters are consistent across the different data sets.}

\rev{In contrast, the standard reference set Ref.~\#1--\#3 provides only weak information on the yaw inertia $J_\mathrm{z}$, as evidenced by the large standard deviation of $1473\,\mathrm{kg\,m^2}$. Extending the reference data to Ref.~\#1--\#5 by adding the two further slalom maneuvers improves the identification of $J_\mathrm{z}$ noticeably, but at the cost of doubling the measurement time to $40\,\mathrm{s}$. \revv{The optimized input set Opt.~\#1--\#2 achieves a standard deviation of only $585\,\mathrm{kg\,m^2}$ with the same measurement time of $20\,\mathrm{s}$ as Ref.~\#1--\#3, corresponding to a reduction of $60.3\%$.}}
\begin{figure}[tb]
	\centering
	\footnotesize
	\vspace{1.5ex}
	\includegraphics[width=\linewidth]{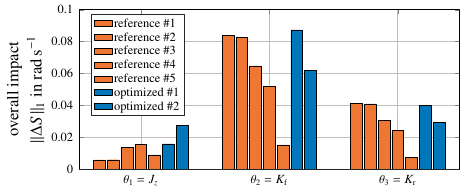}
	\mycaption{\rev{Ranking of maneuver sensitivities}}{\rev{Comparison of the overall impact score $\|\Delta S\|_1$ for the reference maneuvers (orange) and optimized maneuvers (blue) from Fig.~\ref{fig:measurement_overview} with respect to the inertia $\theta_1=J_z$, and the cornering stiffnesses $\theta_2=K_\mathrm{f}$, $\theta_3=K_\mathrm{r}$.}}
	\label{fig:fig_ranking_maneuvers}
\end{figure}
\rev{The improved identification performance of the optimized maneuvers can be explained by their significantly higher sensitivity with respect to the inertia $J_\mathrm{z}$, which is the parameter of interest. This is illustrated in Fig.~\ref{fig:fig_ranking_maneuvers}, where the overall impact score $\|\Delta S\|_1$ for each maneuver is compared across the reference and optimized maneuvers. The optimized maneuvers (blue) exhibit substantially higher sensitivity for $J_\mathrm{z}$ compared to the reference maneuvers (orange), while maintaining comparable sensitivities for the cornering stiffnesses $K_\mathrm{f}$ and $K_\mathrm{r}$. 
\revv{This is consistent with the comparable uncertainty levels of $K_\mathrm{f}$ and $K_\mathrm{r}$ reported in Table~\ref{tab:ident_results_vehicle}.
} 
\revv{Quantitatively, for $J_\mathrm{z}$ the mean impact score increases from $0.0100$ over Ref.~\#1--\#5 to $0.0216$ over Opt.~\#1--\#2 (about $+116\%$).
} This confirms that the optimal excitation design successfully tailored the maneuvers to enhance the identifiability of the target parameter, leading to improved estimation quality within the same measurement time.}

\rev{It is worth emphasizing that, beyond the design of new tailored excitations, the parameter-wise sensitivity signatures depicted in Fig.~\ref{fig:fig_ranking_maneuvers} are already a valuable analysis tool for \emph{existing} measurement catalogues in practice. By computing the impact score $\|\Delta S\|_1$ for each maneuver and parameter, redundant maneuvers carrying similar information can be identified and removed, gaps in the parameter coverage of a catalogue can be detected, and individual maneuvers can be ranked according to their identification value per measurement time. The proposed framework can therefore be used both to generate new optimal maneuvers and to systematically assess and curate established field-test catalogues.}
\section{Conclusion and Outlook}
\label{sec:Conclusion}
This paper presents a systematic framework for optimal excitation design in parameter identification, formulated as a stochastic optimal control problem that maximizes global sensitivity measures while accounting for model uncertainties and measurement noise. It develops two complementary uncertainty quantification approaches to make the computationally intractable problem tractable. For systems with known mathematical structure, the intrusive polynomial chaos expansion method constructs deterministic surrogate models that enable rapid sensitivity evaluation through a single simulation, achieving significant computational speedups. For black-box models, a novel non-intrusive approach based on optimal transport theory using Wasserstein distances requires no knowledge of internal system dynamics.

The vehicle dynamics application demonstrated that both methods generate comparable optimal excitations, with the intrusive method providing substantial computational advantages when applicable. The framework successfully produces parameter-specific excitations tailored to identification objectives.
\rev{\revv{The methodology was further validated experimentally on a test vehicle, where the optimized maneuvers reduced the standard deviation of the yaw-inertia estimate by more than a factor of two compared to a standard reference set of equal measurement time, and even outperformed a reference set of twice the duration.} \revv{This confirms that, for the considered vehicle lateral-dynamics identification problem, the proposed framework translates theoretical sensitivity gains into tangible improvements in identification quality per measurement time under real-world conditions.}}

Future work includes the integration with online parameter identification algorithms and the extension to multi-output and constrained-output settings. The methodology has broad applicability across engineering disciplines, providing a foundation for advancing virtual engineering processes and reducing reliance on extensive experimental programs as systems become increasingly complex and expensive to test.

\footnotesize
\balance
\setlength{\bibsep}{0pt plus 0.3ex}
\bibliographystyle{unsrt}
\bibliography{references}

\end{document}